\documentclass[11pt]{amsart}

\usepackage{amsmath,amssymb,amsthm}
\usepackage{thmtools,enumerate}

\usepackage[german,english]{babel}
\usepackage[autostyle]{csquotes}

\usepackage{tikz-cd}
\usetikzlibrary{babel}

\usepackage{hyperref}
\hypersetup{
	colorlinks=true,
	linkcolor=red,
	citecolor=blue}

\theoremstyle{plain}

\newtheorem{thm}{Theorem}[section]
\newtheorem{prop}[thm]{Proposition}
\newtheorem{lem}[thm]{Lemma}
\newtheorem{cor}[thm]{Corollary}
\newtheorem{thmapp}[thm]{Theorem}

\theoremstyle{definition}

\newtheorem{dfn}[thm]{Definition}
\newtheorem{rem}[thm]{Remark}
\newtheorem{notation}[thm]{Notation}

\newcommand{\Z}{\mathbb{Z}}
\newcommand{\N}{\mathbb{N}}
\newcommand{\C}{\mathbb{C}}
\newcommand{\R}{\mathbb{R}}
\newcommand{\Q}{\mathbb{Q}}
\newcommand{\OO}{\mathcal{O}}

\DeclareMathOperator{\Exc}{Exc}
\DeclareMathOperator{\mult}{mult}
\DeclareMathOperator{\Supp}{Supp}
\DeclareMathOperator{\Nklt}{Nklt}

\begin{document}
	
	\title[Special MMP for log canonical generalised pairs]{Special MMP for\\ log canonical generalised pairs}
	
	\author{Vladimir Lazi\'c}
	\address{\foreignlanguage{german}{Fachrichtung Mathematik, Campus, Geb\"aude E2.4, Universit\"at des Saarlandes, 66123 Saarbr\"ucken, Germany}}
	\email{lazic@math.uni-sb.de}
	
	\author[Nikolaos Tsakanikas]{Nikolaos Tsakanikas\\ {\MakeLowercase{with an appendix joint with} Xiaowei Jiang}}
	\address{\foreignlanguage{german}{Fachrichtung Mathematik, Campus, Geb\"aude E2.4, Universit\"at des Saarlandes, 66123 Saarbr\"ucken, Germany}}
	\email{tsakanikas@math.uni-sb.de}
	
	\address{Department of Mathematics, Tsinghua University, Hai Dian District, Beijing, China 100084}
	\email{jxw20@mails.tsinghua.edu.cn}
	
	\thanks{Lazi\'c gratefully acknowledges support by the Deutsche Forschungsgemeinschaft (DFG, German Research Foundation) -- Project-ID 286237555 -- TRR 195. We would like to thank Jihao Liu and Kenta Hashizume for useful discussions related to this work, and Xiaowei Jiang for pointing out that an assumption in lower dimensions, which was present in a previous version of this paper, can be removed from Lemma \ref{lem:lifting_scaling_g} and from Theorems \ref{thm:lifting_scaling_g} and \ref{thm:HL18_term_with_scaling}. We would also like to thank the referee for useful comments and suggestions. \newline
	}
	
	\begin{abstract}
		We show that minimal models of $\Q$-factorial NQC log ca\-no\-ni\-cal generalised pairs exist, assuming the existence of minimal models of smooth varieties. More generally, we prove that on a $\Q$-factorial NQC log canonical generalised pair $ (X,B+M) $ we can run an MMP with scaling of an ample divisor which terminates, assuming that it admits an NQC weak Zariski decomposition or that $ K_X+B+M$ is not pseudoeffective. As a consequence, we establish several existence results for minimal models and Mori fibre spaces. 
	\end{abstract}

	\maketitle
	
	\begingroup
		\hypersetup{linkcolor=black}
		\setcounter{tocdepth}{1}
		\tableofcontents
	\endgroup

	\section{Introduction}
	
	This paper addresses the problems of the existence of minimal models and Mori fibre spaces for generalised pairs. The category of generalised pairs, which enlarges the category of usual pairs, was introduced in \cite{BZ16} and was instrumental in several recent developments in birational geometry. A partial overview can be found in \cite{Bir21c}, with further developments in \cite{CT20,Hash20b,HanLiu20,HaconLiu21,HanLiu21,LT22} and the references therein. As indicated in \cite{Mor18,CT20,LT22}, understanding generalised pairs is essential even if one is only interested in studying the birational geometry of varieties.
	
	Generalised pairs are, roughly speaking, couples of the form $(X,B+M)$, where $X$ is a normal projective variety, $B$ is an effective $\R$-divisor on $X$ and $M$ is an auxiliary $\R$-divisor on $ X $ such that $ K_X+B+M $ is $ \R $-Cartier. Here, the divisor $M$ has certain positivity properties and satisfies the additional \emph{NQC condition}, which allows for generalised pairs to behave similarly to usual pairs in many proofs; for the precise definitions, see Section \ref{section:Preliminaries}.
	
	\medskip
	
	In this paper we extend our earlier results from \cite{LT22} to the setting of generalised pairs and we generalise several other results from the context of usual pairs to the category of generalised pairs. The following is our main result.
		
	\begin{thm}\label{thm:mainthm}
		Assume the existence of minimal models for smooth varieties of dimension $n-1$.
		
		Let $ (X/Z,B+M) $ be a $ \Q $-factorial NQC log canonical generalised pair of dimension $ n $. Assume that either
		\begin{enumerate}[\normalfont (a)]
			\item $ (X,B+M) $ admits an NQC weak Zariski decomposition over $Z$, or 
		
			\item $ K_X+B+M$ is not pseudoeffective over $ Z $.
		\end{enumerate}
		Then for any $\R$-divisors $ P $ and $ N\geq0 $ on $ X $ such that $ P $ is the pushforward of an NQC divisor over $Z$ on a birational model of $ X $ and such that the generalised pair $ \big( X, (B+N) + (M+P) \big) $ is log canonical and the divisor $ K_X + B + N + M + P $ is nef over $ Z $, there exists a $ (K_X + B + M) $-MMP over $Z$ with scaling of $P+N$ that terminates.
		
		In particular, the generalised pair $ (X,B+M) $ has either a minimal model over $Z$ or a Mori fibre space over $ Z $.
	\end{thm}
	
	For the definition of an NQC weak Zariski decomposition, see Section \ref{section:Preliminaries}. Note that in this and in all other results of this paper, the assumption in lower dimensions means the existence of \emph{relative} minimal models, that is, minimal models of smooth quasi-projective varieties which are projective and whose canonical class is pseudoeffective over another normal quasi-projective variety.

	Theorem \ref{thm:mainthm} shows in particular that, assuming the existence of minimal models for smooth varieties of dimension $n-1$, the existence of an NQC weak Zariski decomposition of a $\Q$-factorial NQC log canonical generalised pair is equivalent to the existence of a minimal model of that generalised pair, see Remark \ref{rem:MMimplWZD}. This generalises \cite[Theorem 1.5]{Bir12b} to the setting of generalised pairs with significantly refined assumptions in lower dimensions, cf.\ \cite[Theorem B]{LT22}.
	
	\medskip
	
	After this paper appeared on the arXiv, Hashizume \cite{Hash22} and Liu and Xie \cite{LX22} obtained new results regarding the existence of minimal models in the sense of Birkar-Shokurov for generalised pairs with boundaries containing ample divisors. This allows us to improve considerably Proposition \ref{prop:HH19_prop_g} below by removing the assumption in lower dimensions. In particular, we establish the existence of Mori fibre spaces for non-pseudoeffective generalised pairs in any dimension in the following Theorem \ref{thm:Mfs}. The proof of this result is given in the appendix, which is written jointly with Xiaowei Jiang.
	
	\begin{thm}\label{thm:Mfs}
		Let $ (X/Z,B+M) $ be a $\Q$-factorial NQC log canonical generalised pair. Assume that $ (X,B+M) $ has a minimal model in the sense of Birkar-Shokurov over $ Z $ or that $ K_X+B+M $ is not pseudoeffective over $ Z $. Let $ A $ be an effective $ \R $-Cartier $ \R $-divisor on $ X $ which is ample over $ Z $ such that $ K_X + B + A + M $ is nef over $Z$. Then there exists a $ (K_X + B + M) $-MMP over $Z$ with scaling of $A$ that terminates.
		
		In particular:
		\begin{enumerate}[\normalfont (a)]
			\item $(X,B+M)$ has a minimal model in the sense of Birkar-Shokurov over $Z$ if and only if it has a minimal model over $Z$,
			
			\item if $K_X+B+M$ is not pseudoeffective over $Z$, then $(X,B+M)$ has a Mori fibre space over $Z$.
		\end{enumerate}
	\end{thm}
	
	Theorem \ref{thm:mainthm} has several immediate consequences. First of all, it allows us to reduce the problem of the existence of minimal models for $\Q$-factorial NQC log canonical generalised pairs to the problem of the existence of minimal models for smooth varieties, cf.\ \cite[Theorem A]{LT22}.
	
	\begin{cor} \label{cor:MM_impl_g}
		The existence of minimal models for smooth varieties of dimension $n$ implies the existence of minimal models for $\Q$-factorial NQC log canonical generalised pairs of dimension $n$.
	\end{cor}

	We also have the following corollary in low dimensions, cf.\ \cite[Corollary D]{LT22}.
	
	\begin{cor}\label{cor:maincor}
		Let $ (X/Z,B+M) $ be a $ \Q $-factorial NQC log canonical generalised pair of dimension $ n $. Then the following statements hold:
		\begin{enumerate}
			\item If $n \leq 4$ and $ K_X+B+M$ is pseudoeffective over $ Z $, then $ (X,B+M) $ has a minimal model over $Z$.
			
			\item If $n \leq 5$, $ K_X+B+M$ is pseudoeffective over $ Z $ and a general fibre of the morphism $ X\to Z $ is uniruled, then $ (X,B+M) $ has a minimal model over $Z$.
		\end{enumerate}
	\end{cor}
	
	Part (i) of Corollary \ref{cor:maincor} was already shown in \cite[Theorem 1.5]{HaconLiu21} as a consequence of \cite[Theorems 1.1 and 1.2]{CT20}. Part (ii) of Corollary \ref{cor:maincor} is a special case of the following generalisation of \cite[Theorems C and 4.3]{LT22}.
	
	\begin{thm}\label{thm:MM_uniruled_g}
		Assume the existence of minimal models for smooth varieties of dimension $n-1$.
		
		Let $ (X/Z,B+M) $ be a pseudoeffective $\Q$-factorial NQC log canonical generalised pair of dimension $ n $ such that a general fibre of the morphism $ X\to Z $ is uniruled. Then $(X,B+M)$ has a minimal model over $Z$.
	\end{thm}
	
	In our previous work \cite{LT22} we obtained similar statements to Corollary \ref{cor:MM_impl_g} and Theorem \ref{thm:MM_uniruled_g} when additionally the underlying variety of the given generalised pair has klt singularities. The main reason for this additional assumption was that, in its presence, one can reduce several foundational results in the geometry of generalised pairs, such as the existence of surgery operations in the MMP, to statements about usual pairs, see \cite{BZ16,HanLi18}. These foundational results (the Cone and Contraction theorems, the existence of divisorial contractions and flips) were very recently established in \cite{HaconLiu21} for $\Q$-factorial NQC log canonical generalised pairs.
	
	Finally, we obtain the following generalisation of \cite[Corollary 1.6]{Bir12b}.
	
	\begin{cor}
		\label{cor:MM_bigger_boundary_g}
		Assume the existence of minimal models for smooth varieties of dimension $n-1$.
		
		Let $(X/Z,B_1+M_1)$ and $(X/Z,B_2+M_2)$ be $ \Q $-factorial NQC log canonical generalised pairs of dimension $n$ such that $B_2 \geq B_1 $ and $ M_2 - M_1 $ is the pushforward of an NQC divisor on some higher model of $ X $. If $(X,B_1+M_1)$ admits an NQC weak Zariski decomposition over $Z$, then $(X,B_2+M_2)$ has a minimal model over $Z$.
	\end{cor}
	
	We now outline the main ingredients in the proof of Theorem \ref{thm:mainthm}. First, under some mild assumptions in lower dimensions, we prove the existence of NQC weak Zariski decompositions for generalised pairs satisfying additional properties by using essentially the same line of arguments as in the proof of \cite[Theorem 3.2]{LT22}, see Theorem \ref{thm:WZD_existence_g}. In these arguments, we use the MMP in order to find an appropriate birational model on which we can \enquote{lift} an NQC weak Zariski decomposition from a lower-dimensional variety by the canonical bundle formula for generalised pairs \cite{Fil20,HanLiu21}. Second, in Theorem \ref{thm:HL18_term_with_scaling} we generalise \cite[Theorem 4.1(iii)]{Bir12a} and \cite[Theorem 4.1]{HanLi18} to $\Q$-factorial NQC log canonical generalised pairs, namely we show that under reasonable assumptions every MMP with scaling terminates. Next, we employ this result in Section \ref{section:Proofs} to demonstrate that the existence of minimal models in the sense of Birkar-Shokurov implies the termination of flips with scaling of ample divisors in the category of generalised pairs, following the same strategy as in \cite{HH20}, but using crucially Theorems \ref{thm:WZD_existence_g} and \ref{thm:HL18_term_with_scaling}. Here, the central beautiful idea from \cite{HH20} is to a priori choose carefully an ample divisor with which one runs the MMP with scaling. All these results, together with \cite[Theorem 4.4]{LT22}, which plays a fundamental role in this paper, allow us now to deduce Theorem \ref{thm:mainthm}.

	\section{Preliminaries}
	\label{section:Preliminaries}
	
	Throughout the paper we work over the field $ \C $ of complex numbers. Unless otherwise stated, we assume that varieties are normal and quasi-projective and that a variety $X$ over a variety $Z$, denoted by $X/Z$, is projective over $Z$. We often quote in the paper the Negativity lemma \cite[Lemma 3.39(1)]{KM98}.
	
	A \emph{fibration} is a projective surjective morphism with connected fibres. A \emph{birational contraction} is a birational map whose inverse does not contract any divisors.
	
	Let $ \pi \colon X\to Z $ be a projective morphism between normal varieties and let $ D $ be an $ \R $-Cartier $ \R $-divisor on $ X $. The divisor $ D $ is \emph{pseudoeffective over $ Z $} if it is pseudoeffective on a very general fibre of $ \pi $, and $ D $ is an \emph{NQC divisor (over $ Z $)} if it is a non-negative linear combination of $\Q$-Cartier divisors on $X$ which are nef over $ Z $; the acronym NQC stands for \emph{nef $\Q$-Cartier combinations} \cite{HanLi18}. Two $ \R $-Cartier $ \R $-divisors $ D_1 $ and $ D_2 $ on $ X $ are \emph{$\R$-linearly equivalent over $ Z $}, denoted by $ D_1 \sim_{\R,Z} D_2 $, if there exists an $\R$-Cartier $\R$-divisor $G$ on $Z$ such that $D_1\sim_\R D_2+\pi^*G$, and they are \emph{numerically equivalent over $ Z $}, denoted by $ D_1 \equiv_Z D_2 $, if $ D_1\cdot C = D_2 \cdot C $ for any curve $ C $ contained in a fibre of $ \pi $.

	Given a normal projective variety $ X $ and a pseudoeffective $\R$-Cartier $\R$-divisor $ D $ on $X$, we denote by $ \nu(X,D) $ \emph{the numerical dimension of $ D $}, see \cite[Chapter V]{Nak04}.

	Let $\varphi \colon X\dashrightarrow Y$ be a birational contraction between normal varieties, let $D$ be an $\R$-Cartier $\R$-divisor on $X$ and assume that $\varphi_*D$ is $\R$-Cartier. We say that the map $\varphi$ is \emph{$D$-nonpositive} (respectively \emph{$D$-negative}) if there exists a resolution of indeterminacies $(p,q)\colon W\to X\times Y$ of $\varphi$ such that $ W $ is smooth and  
	\[ p^*D\sim_\R q^* \varphi_*D + E , \]
	where $E$ is an effective $q$-exceptional $\R$-divisor (respectively $E$ is an effective  $q$-exceptional $\R$-divisor whose support contains the strict transform of every $\varphi$-exceptional prime divisor).

	\subsection{Generalised pairs}
	
	For the theory of usual pairs, their singularities and the Minimal Model Program (MMP) we refer to \cite{KM98,Fuj17}. For the analogous concepts and results in the setting of generalised pairs we refer to \cite{BZ16,HanLi18,CT20,LMT,HaconLiu21,LT22} and the relevant references therein.
	
	Here, we briefly recall the definitions of generalised pairs and their various classes of singularities as well as some basic properties that will be needed in this paper.
	
	\begin{dfn}
		A \emph{generalised pair}, abbreviated as \emph{g-pair}, consists of a normal variety $ X $, equipped with projective morphisms
		\[ X' \overset{f}{\longrightarrow} X \longrightarrow Z , \]
		where $ f $ is birational and $ X' $ is a normal variety, an effective $ \R $-divisor $ B $ on $X$, and an $\R$-Cartier $\R$-divisor $ M' $ on $X'$ which is nef over $Z$, such that $ K_X + B + M $ is $ \R $-Cartier, where $ M := f_* M' $. We say that the divisor $ B $, respectively $ M $, is the \emph{boundary part}, respectively the \emph{nef part}, of the g-pair. 
		
		Furthermore, we say that $ (X/Z,B+M) $ is \emph{NQC} if $M'$ is an NQC divisor on $ X' $, and that it is \emph{pseudoeffective over $Z$} if $ K_X + B + M $ is pseudoeffective over $Z $.
	\end{dfn}
	
	The variety $X'$ in the definition may always be chosen as a sufficiently high birational model of $ X $. Often in this paper we do not refer explicitly to the data of a g-pair and write simply $(X/Z,B+M)$, but remember the whole g-pair structure.
	
	We make the following convention throughout the paper: when we write $\big(X,(B+N)+(M+P)\big)$ for a $\Q$-factorial g-pair, then it is implicitly assumed that $(X,B+M)$ and $(X,N+P)$ are likewise g-pairs.
		
	\begin{dfn}
		Let $ (X,B+M) $ be a g-pair with data $ X' \overset{f}{\to} X \to Z $ and $ M' $. Let $ E $ be a divisorial valuation over $ X $; its \emph{centre} on $ X $ is denoted by $ c_X(E) $. We may assume that $ c_X(E) $ is a prime divisor on $ X' $. If we write 
		\[ K_{X'} + B' + M' \sim_\R f^* ( K_X + B + M ) \]
		for some $ \R $-divisor $ B' $ on $ X' $, then the \emph{discrepancy of $ E $ with respect to $ (X,B+M) $} is 
		\[ a(E, X, B+M) :=-\mult_E B' . \]
		
		We say that the g-pair $ (X,B+M) $ is:
		\begin{enumerate}
			\item[(a)] \emph{klt} if $  a(E,X,B+M) > -1 $ for any divisorial valuation $ E $ over $ X $,
			
			\item[(b)] \emph{log canonical} if $  a(E,X,B+M) \geq -1 $ for any divisorial valuation $ E $ over $ X $,
			
			\item[(c)] \emph{dlt} if it is log canonical and if there exists an open subset $U \subseteq X$ such that $(U,B|_U)$ is a log smooth pair, and if $ a(E,X,B+M) = -1 $ for some divisorial valuation $ E $ over $X$, then $ c_X(E) \cap U \neq \emptyset $ and $ c_X(E) \cap U $ is a log canonical centre of $(U,B|_U)$.
		\end{enumerate}
	\end{dfn}
	
	We have adopted the definition of dlt singularities from \cite{HanLi18}. 
	
	The following lemma generalises \cite[Corollaries 2.35(1) and 2.39(1)]{KM98} to the setting of g-pairs.
		
	\begin{lem}\label{lem:sing_smaller_boundary}
		If $\big(X/Z,(B+N)+(M+P)\big)$ is a $\Q$-factorial klt, respectively dlt, log canonical, g-pair, then the g-pair $(X,B+M)$ is also klt, respectively dlt, log canonical.
	\end{lem}
	
	\begin{proof}
		If $ \big(X,(B+N)+(M+P)\big) $ is klt or log canonical, then the statement was shown in \cite[Lemma 2.6]{CT20}. We may thus assume that the pair is dlt. By definition there exists an open subset $ U \subseteq X $ such that $\big(U,(B+N)|_U\big)$ is a log smooth pair, and if $ a\big(E,X,(B+N)+(M+P)\big) = -1 $ for some divisorial valuation $ E $ over $X$, then $ c_X(E) \cap U \neq \emptyset $ and $ c_X(E) \cap U $ is a log canonical centre of $\big(U,(B+N)|_U\big)$. 
		
		Note that the g-pair $(X,B+M)$ is log canonical by the above. Fix a divisorial valuation $E$ over $X$ with 
		$$ a(E,X,B+M) = -1 .$$
		It suffices to show that $ c_X(E) \cap U \neq \emptyset $ and that $ c_X(E) \cap U $ is a log canonical centre of $(U,B|_U)$. 
		
		Pick a sufficiently high log resolution $f\colon X'\to X$ of $(X,B+N)$ such that there exist $\R$-divisors $M'$ and $P'$ on $X'$ which are nef over $Z$ with $M=f_*M'$ and $P=f_*P'$, and such that $c_{X'}(E)$ is a divisor. By the Negativity lemma we have
		$$ f^* P = P' + F , $$
		where $ F $ is an effective $ f $-exceptional $\R$-divisor, and an easy calculation then gives
		\begin{align*}
		-1 &= a(E,X,B+M) \\
		&= a\big(E, X, (B+N) + (M+P)\big) + \mult_E \big( f^* N + F \big) \geq -1.
		\end{align*}
		This implies 
		$$ a\big(E, X, (B+N) + (M+P)\big) = -1 \quad \text{and}\quad  c_X(E) \nsubseteq \Supp N .$$
		The first relation shows that $ c_X(E) \cap U \neq \emptyset $ is a log canonical centre of $\big(U,(B+N)|_U\big)$ by the definition of $U$, hence it is a log canonical centre of $(U,B|_U)$ by the second relation, as desired.
	\end{proof}
	
	The next result is \cite[Proposition 3.9]{HanLi18} and will be used several times in the paper without explicit mention.
	
	\begin{lem}
		Let $(X,B+M)$ be a log canonical g-pair with data $ \widetilde X \overset{f}{\to} X \to Z $ and $ \widetilde M $. Then, after possibly replacing $f$ with a higher model, there exist a $\Q$-factorial dlt g-pair $(X',B'+M')$ with data $ \widetilde X \overset{g}{\to} X' \to Z $ and $ \widetilde M $, and a projective birational morphism $h \colon X' \to X$ such that 
		\[ K_{X'} + B' + M' \sim_\R h^*(K_X + B +M)\quad\text{and}\quad B' = h_*^{-1} B + E , \]
		where $ E $ is the sum of the $ h $-exceptional prime divisors. The g-pair $(X',B'+M')$ is called a \emph{dlt blowup} of $(X,B+M)$.
	\end{lem}

	\subsection{Minimal models and Mori fibre spaces}
	
	We recall the definitions of minimal models \emph{in the usual sense} and \emph{in the sense of Birkar-Shokurov}, as well as the definition of Mori fibre spaces.
	
	\begin{dfn}\label{dfn:1}
		Assume that we have a birational map $\varphi \colon X \dashrightarrow X'$ over $ Z $ and g-pairs $(X/Z,B+M)$ and $(X'/Z,B'+M')$ such that $(X,B+M)$ is log canonical and the divisors $ M $ and $ M' $ are pushforwards of the same nef $\R$-divisor on a common birational model of $X$ and $X'$. 
		\begin{enumerate}[(a)]
			\item The map $\varphi$ is a \emph{minimal model in the sense of Birkar-Shokurov over $Z$} of the g-pair $(X,B+M)$ if $ B' =\varphi_*B+E$, where $E$ is the sum of all prime divisors which are contracted by $\varphi^{-1}$, if $X'$ is $\Q$-factorial, if the divisor $K_{X'}+B'+M'$ is nef over $Z$ and if
			$$a(F,X,B+M) < a(F,X',B'+M')$$
			for any prime divisor $ F $ on $ X $ which is contracted by $\varphi $. Note that the g-pair $ (X',B'+M')$ is log canonical by \cite[Lemma 2.8(i)]{LMT}.
			
			If, moreover, the map $\varphi$ is a birational contraction, but $X'$ is not necessarily $\Q$-factorial if $X$ is not $\Q$-factorial (and $X'$ is $\Q$-factorial if $X$ is $\Q$-factorial), then $\varphi$ is a \emph{minimal model of $(X,B+M)$ over $ Z $}.
			
			\item If the map $\varphi$ is a birational contraction, if $ B' = \varphi_* B $, if $ K_{X'} + B' + M' $ is ample over $Z$ and if 
			$$a(F,X,B+M)\leq a(F,X',B'+M')$$
			for every $\varphi$-exceptional prime divisor $F$ on $X$, then $(X',B'+M')$ is a \emph{canonical model of $(X,B+M)$ over $ Z $}. Note that the g-pair $ (X',B'+M')$ is log canonical by \cite[Lemma 2.8(i)]{LMT} and unique up to isomorphism by \cite[Lemma 2.12]{LMT}.
			
			\item If the map $\varphi$ is a birational contraction, if $ B' = \varphi_* B $, if $X'$ is $\Q$-factorial when $X$ is $\Q$-factorial, if there exists a $ (K_{X'} + B' + M') $-negative extremal contraction $ X' \to T $ over $ Z $ with $ \dim X' > \dim T $, if for any divisorial valuation $ F $ over $ X $ we have 
			\[ a(F,X,B+M) \leq a(F,X',B'+M') , \]
			and the strict inequality holds if $ c_X(F) $ is a $ \varphi $-exceptional prime divisor, then $ (X',B'+M') $ is a \emph{Mori fibre space of $ (X,B+M) $ over $ Z $}.
		\end{enumerate}		
	\end{dfn}

	\begin{rem}\label{rem:1}
		For the differences among these notions of a minimal model, see \cite[Subsection 2.2]{LT22}. We highlight that here we allow a minimal model in the sense of Birkar-Shokurov to be log canonical and not only dlt.
		
		We note here the following properties of minimal models in the sense of Birkar-Shokurov. In the notation from Definition \ref{dfn:1}(a), for every resolution of indeterminacies $(p,q)\colon W \to X\times X'$ of the map $\varphi$, the divisor 
		$$ p^*(K_{X}+B+M) - q^*(K_{X'}+B'+M') $$
		is effective and $q$-exceptional. In particular, we have $a(F,X,B+M) \leq a(F,X',B'+M')$ for every geometric valuation $F$ over $X$ and 
		$$a(F,X,B+M) = a(F,X',B'+M') = -1$$
		for any geometric valuation $F$ over $X$ which is exceptional over $X$ but is not exceptional over $Y$. This is proved analogously as in  \cite[Remark 2.6]{Bir12a}.
		
		Furthermore, if $B_W$ is the sum of $p_*^{-1}B$ and of all $p$-exceptional prime divisors on $W$, then the divisor
		$$ K_W+B_W+M_W - p^*(K_X+B+M) $$
		is effective and $p$-exceptional since the g-pair $(X,B+M)$ is log canonical, where $M_W$ is the pushforward of the nef $\R$-divisor as in Definition \ref{dfn:1}. Moreover, it is also $q$-exceptional: indeed, if $F$ is a prime divisor on $W$ which is $p$-exceptional but not $q$-exceptional, then $a(F,X,B+M) = -1$ by the previous paragraph, so $F$ cannot be a component of $ K_W+B_W+M_W - p^*(K_X+B+M) $.
	\end{rem}

	\subsection{NQC weak Zariski decompositions}
	
	First, we recall the notion of an NQC weak Zariski decomposition, which plays a fundamental role in this paper. For further information we refer to \cite{Bir12b,HanLi18,LT22}.
	
	\begin{dfn}
		Let $ X \to Z $ be a projective morphism between normal varieties and let $ D $ be an $ \R $-Cartier $\R$-divisor on $ X $. An \emph{NQC weak Zariski decomposition of $ D $ over $Z$} consists of a projective birational morphism $ f \colon W \to X $ from a normal variety $ W $ and a numerical equivalence $ f^* D \equiv_Z P + N $, where $ P $ is an NQC divisor (over $ Z $) on $ W $ and $ N $ is an effective $\R$-Cartier $\R$-divisor on $W$.
		
		Moreover, if $ (X/Z,B+M) $ is an NQC g-pair, then we say that $ (X,B+M) $ \emph{admits an NQC weak Zariski decomposition over $Z$} if the divisor $ K_X + B + M $ admits an NQC weak Zariski decomposition over $Z$.
	\end{dfn}
	
	\begin{rem}\label{rem:MMimplWZD}~
		If an NQC log canonical g-pair $ (X,B+M) $ has a minimal model in the usual sense or in the sense of Birkar-Shokurov over $Z$, then it admits an NQC weak Zariski decomposition over $Z$ by (the proof of) \cite[Proposition 5.1]{HanLi18}.
	\end{rem}
	
	Next, we prove an easy corollary of \cite[Theorem 4.4]{LT22}, which was hinted at but not formulated in \cite{LT22} and which will be used crucially in the proof of Lemma \ref{lem:MM_existence_HH19}. Note that this result generalises \cite[Corollary 1.6]{Bir12b} to the setting of g-pairs with significantly weaker assumptions in lower dimensions.
	
	\begin{lem}
		\label{lem:MM_bigger_boundary_g}
		Assume the existence of minimal models for smooth varieties of dimension $n-1$.
		
		Let $(X/Z,B_1+M_1)$ and $(X/Z,B_2+M_2)$ be NQC log canonical g-pairs of dimension $n$ such that $B_2 \geq B_1 $ and $ M_2 - M_1 $ is the pushforward of an NQC divisor on some higher model of $ X $. If $(X,B_1+M_1)$ admits an NQC weak Zariski decomposition over $Z$, then $(X,B_2+M_2)$ has a minimal model in the sense of Birkar-Shokurov over $Z$.
	\end{lem}
	
	\begin{proof}
		Let $ f \colon X' \to X $ be a high birational model such that there exists an NQC divisor $ Q' $ on $ X' $ with
		\begin{equation} \label{eq:1_bigger_boundary}
			Q := f_* Q' = M_2 - M_1 ,
		\end{equation}
		and such that
		\begin{equation} \label{eq:2_bigger_boundary}
			f^* (K_X + B_1 + M_1) \equiv_Z P + N ,
		\end{equation}
		where $ P $ is an NQC divisor on $ X' $ and $ N $ is an effective $\R$-Cartier $ \R $-divisor on $ X' $, see \cite[Remark 2.11]{LT22}. By the Negativity lemma we have
		\begin{equation} \label{eq:3_bigger_boundary}
			f^*Q = Q' + E ,
		\end{equation}
		where $ E $ is an effective $ f $-exceptional $ \R $-Cartier $ \R $-divisor on $ X' $. Moreover, since $ B_1 \leq B_2 $, we may write
		\begin{equation} \label{eq:4_bigger_boundary}
			B_2 = B_1 + G ,
		\end{equation}
		where $ G $ is an effective $ \R $-divisor on $ X $. Then \eqref{eq:1_bigger_boundary}, \eqref{eq:2_bigger_boundary},  \eqref{eq:3_bigger_boundary} and \eqref{eq:4_bigger_boundary} give
		\[ f^* (K_X + B_2 + M_2) \equiv_Z  (P+Q') + (N + f^*G + E) , \]
		which is an NQC weak Zariski decomposition of $ (X,B_2+M_2) $ over $Z$. We conclude by \cite[Theorem 4.4(i)]{LT22}.
	\end{proof}
	
	Finally, for a variant of Lemma \ref{lem:MM_bigger_boundary_g}, where we assume additionally that $ X $ is $ \Q $-factorial, but we obtain the stronger conclusion that $ (X,B_2+M_2) $ has a minimal model in the usual sense over $ Z $, see Corollary \ref{cor:MM_bigger_boundary_g}.

	\subsection{The MMP for generalised pairs}
	
	In this paper we use -- frequently without explicit mention -- the foundations of the Minimal Model Program for NQC $\Q$-factorial log canonical generalised pairs, as established recently in \cite{HaconLiu21}. We recall briefly the main results.
	
	Let $ (X/Z,B+M) $ be a $ \Q $-factorial NQC log canonical g-pair. First of all, the Cone and Contraction theorems for $ (X/Z,B+M) $ were established in \cite[Theorem 1.3]{HaconLiu21} and are analogous to the Cone theorem for usual pairs \cite[Theorem 4.5.2]{Fuj17}. Then, similarly as in the case of usual pairs, an extremal contraction can be either divisorial, flipping, or define a Mori fibre space structure. As shown in \cite[Subsection 5.5]{HaconLiu21}, the divisorial and Mori contractions behave well in any MMP; in particular, the Picard number drops after a divisorial contraction. Moreover, flips for $ \Q $-factorial NQC log canonical g-pairs exist by \cite[Theorem 1.2]{HaconLiu21}. All this implies that one may run a $ (K_X + B + M) $-MMP over $Z$, whose termination is not known in general. However, the paper \cite{CT20} establishes the termination of flips for $ \Q $-factorial NQC log canonical g-pairs of dimension $ 3 $ and for pseudoeffective $ \Q $-factorial NQC log canonical g-pairs of dimension $ 4 $, see also \cite{Mor18,HM20}. Note that $\Q$-factoriality is preserved in any MMP by \cite[Corollaries 5.20 and 5.21 and Theorem 6.3]{HaconLiu21}.
	
	Now, let $ \big( X/Z, (B+N) + (M+P) \big) $ be a $ \Q $-factorial NQC log canonical g-pair such that the divisor $ K_X + B + N + M + P $ is nef over $ Z $. Then by \cite[Corollary 5.22]{HaconLiu21} we may run a $ (K_X + B + M) $-MMP with scaling of $ P+N $ over $Z$, cf.\ \cite[Section 3.4]{HanLi18}, whose termination is not known in general. In particular, we may run a $ (K_X + B + M) $-MMP with scaling of an ample divisor over $Z$, cf.\ \cite[Section 4]{BZ16}, \cite[Section 3.1]{HanLi18}. We refer to \cite{BZ16,HanLi18,LT22} for various results concerning the termination of the MMP with scaling of an ample divisor in the setting of g-pairs, and to Sections \ref{section:Term_with_Scaling} and \ref{section:Proofs} for further developments. We state here for future reference the following result, which has already appeared implicitly in \cite{HanLi18,LT22}, and we provide the details for the convenience of the reader.
	
	\begin{lem}\label{lem:HL18_term_with_ample_scaling}
		Let $ (X/Z,B+M) $ be an NQC log canonical g-pair such that $ (X,0) $ is $\Q$-factorial klt. Let $ A $ be an effective $ \R $-Cartier $ \R $-divisor on $ X $ which is ample over $ Z $ such that $ \big( X/Z,(B+A)+M \big) $ is log canonical and $ K_X + B + A + M $ is nef over $Z$. If $ (X,B+M) $ has a minimal model in the sense of Birkar-Shokurov over $ Z $, then any $ (K_X + B + M) $-MMP with scaling of $ A $ over $ Z $ terminates.
	\end{lem}
	
	\begin{proof}
		Denote by $ \lambda_i $ the corresponding nef thresholds in the steps of the given MMP and set $ \lambda := \lim\limits_{i\to \infty} \lambda_i $. We distinguish two cases.
		
		Assume first that $ \lambda > 0 $. Then the given MMP is also a $(K_X + B + M + \frac{\lambda}{2} A)$-MMP. By \cite[Lemma 3.5]{HanLi18} there exists a boundary $ \Delta $ on $ X $ such that 
		$K_X + \Delta \sim_{\R,Z} K_X + B + M + \frac{\lambda}{2} A  $, 
		the pair $(X,\Delta)$ is klt and the divisor $ \Delta $ is big over $ Z $. By \cite[Corollary 1.4.2]{BCHM10} the $(K_X+\Delta)$-MMP over $ Z $ with scaling of $ A $ terminates, and therefore the original MMP terminates.
		
		Assume now that $ \lambda = 0 $. By assumption and by \cite[Theorem 4.1]{HanLi18} we deduce that the given MMP terminates.
	\end{proof}
	
	In the remainder of this subsection we prove an analogue of \cite[Proposition 3.2(5)]{Bir11} in the context of g-pairs, which plays a fundamental role in the proofs of Theorems \ref{thm:WZD_existence_g} and \ref{thm:HH19_thm_g} and is also of independent interest. It generalises both \cite[Lemma 3.17]{HanLi18} and \cite[Lemma 2.20]{LT22}.
	
	\begin{notation}\label{notation}
		Let $X$ be a $\Q$-factorial variety, let $F_1 ,\dots , F_k$ be prime divisors on $X$, and for each $ j \in \{ 1, \dots, \ell \} $ let $G_j$ be the pushforward to $X$ of a nef over $Z$ Cartier divisor on some high birational model of $ X $. Consider the vector space
		\[ 	V := \bigoplus_{i=1}^k \R F_i \oplus \bigoplus_{j=1}^\ell \R G_j. \] 
		We write the elements of $ V $ as sums $ N+P $, where $ N \in \bigoplus_{i=1}^k \R F_i $ and $ P \in \bigoplus_{j=1}^\ell \R G_j $. We define a norm on $ V $ as follows: given $ N = n_1 F_1 + \dots + n_k F_k$ and $ P = p_1 G_1 + \dots + p_\ell G_\ell $, we set
		\[ 	\|N+P\|:=	\max_{i,j} \big\{ |n_i|, |p_j|  \big\} . \] 
		If we consider the cone 
		$$V_{\geq0}:= \bigoplus_{i=1}^k \R_{\geq0} F_i \oplus \bigoplus_{j=1}^\ell \R_{\geq0} G_j$$
		in $V$, then for each $N+P\in V_{\geq0}$ we have that $ (X, N+P) $ is a $\Q$-factorial NQC g-pair. Moreover, the set
		$$ \mathcal{L}(V) := \big\{N + P \in V \mid (X, N + P ) \text{ is log canonical} \big\} $$
		is a rational polytope, which may be unbounded; see for instance \cite[Section 3.3]{HanLi18}.
	\end{notation}
			
	\begin{prop}\label{prop:Bir11_prop_g}
		Assume Notation \ref{notation}. Let $ (X/Z,B+M) \in \mathcal L(V)$ such that the divisor $ K_X+B+M $ is nef over $ Z $. Then there exists a positive real number $ \delta $ such that if $ N+P \in \mathcal{L}(V) $ with $ \| (N-B) + (P-M) \| < \delta $, then any $ (K_X+N+P) $-MMP over $ Z $ is $ (K_X+B+M) $-trivial.
	\end{prop}
	
	\begin{proof}
		First, by \cite[Proposition 2.6]{HanLiu20} there exist positive real numbers $r_1,\dots,r_m$ and $\Q$-divisors $B^{(1)},\dots, B^{(m)}$ and $M^{(1)},\dots, M^{(m)}$ such that 
		\[ \sum_{j=1}^m r_j=1, \quad B = \sum_{j=1}^m r_j B^{(j)},\quad M = \sum_{j=1}^m r_j M^{(j)} , \]
		and for each $ j \in \{1, \dots, m\} $, $(X,B^{(j)}+M^{(j)})$ is a log canonical g-pair and the divisor $ K_X + B^{(j)} + M^{(j)} $ is nef over $ Z $. In particular, we have
		\begin{equation}\label{eq:1_Bir11_g}
			K_X + B + M = \sum_{j=1}^m r_j \big(K_X + B^{(j)} + M^{(j)} \big) . 
		\end{equation}
		Let $ r $ be a positive integer such that $ r \big(K_X + B^{(j)} + M^{(j)} \big) $ is Cartier for every $ j \in \{1, \dots, m\} $, and set
		\[ \alpha := \frac{ \min \{ r_1,\dots,r_m \}}{r} > 0 . \]
		
		Since $\mathcal{L}(V)$ is a polytope, there exists a positive real number $\delta$ with the following property: if $ N+P \in \mathcal{L}(V) $ with $ \| (N-B) + (P-M) \| < \delta $, then there exist a point $ N'+P' \in \mathcal{L}(V) $ and a real number
		\begin{equation}\label{eq:9b}
			0 \leq t < \frac{\alpha}{\alpha + 2 \dim X} , 
		\end{equation}
		such that 
		\begin{equation}\label{eq:2_Bir11_g}
			N+P = (1-t)(B+M) + t(N' + P') .
		\end{equation}
	
		Now, we fix $ N+P \in \mathcal{L}(V) $ with $ \| (N-B) + (P-M) \| < \delta $, and we also fix a number $t$ and a divisor $ N'+P' \in \mathcal{L}(V) $ as in \eqref{eq:9b} and \eqref{eq:2_Bir11_g}. Moreover, we run a $ (K_X + N + P) $-MMP over $ Z $:
		$$ (X,N+P) =: (X_1,N_1+P_1) \dashrightarrow (X_2,N_2+P_2) \dashrightarrow \cdots $$
		For every $ i $ denote by $ B_i $, $ B_i^{(j)} $, $ M_i $ and $ M_i^{(j)} $ the pushforwards on $ X_i $ of $ B $, $ B^{(j)} $, $ M $ and $ M^{(j)} $, respectively. Assume that we showed that the MMP is $ (K_X+B+M) $-trivial up to the variety $X_i$. We will now show that the map $X_i\dashrightarrow X_{i+1}$ is $ (K_{X_i}+B_i+M_i) $-trivial.
		
		Denote by $ N_i' $ and $ P_i' $ the pushforwards on $ X_i $ of $ N' $ and $ P' $, respectively. From \eqref{eq:2_Bir11_g} we have
		\[ K_X+N+P = (1-t)(K_X+B+M) + t(K_X+N' + P'), \]
		which implies that the MMP is $ (K_X+N'+P') $-negative up to variety $X_i$. In particular, the g-pair $(X_i,N_i'+P_i')$ is log canonical. 
		
		Since the MMP is $ (K_X+B+M) $-trivial up to variety $X_i$, by \cite[Theorem 1.3(4)(c)]{HaconLiu21} we deduce that the g-pair $(X_i,B_i+M_i)$ is log canonical and that the divisor $K_{X_i}+B_i+M_i$ is nef over $Z$. Consequently, by \eqref{eq:1_Bir11_g}, the g-pairs $(X_i , B_i^{(j)} + M_i^{(j)})$ are log canonical, the divisors $K_{X_i} + B_i^{(j)} + M_i^{(j)}$ are nef over $ Z $, and $r(K_{X_i} + B_i^{(j)} + M_i^{(j)})$ is Cartier for every $ j \in \{1, \dots, m\} $ again by \cite[Theorem 1.3(4)(c)]{HaconLiu21}.
		
		Let $ R $ be the $ (K_{X_i}+N_i+P_i) $-negative extremal ray over $ Z $ contracted at the $i$-th step of the MMP. As the divisor $ K_{X_i}+B_i+M_i $ is nef over $ Z $, by \eqref{eq:2_Bir11_g} we infer that $ (K_{X_i}+N_i'+P_i') \cdot R < 0 $, hence by \cite[Theorem 1.3(2)]{HaconLiu21} there exists a rational curve whose class belongs to $R$ and satisfies
		\begin{equation}\label{eq:9a}
			(K_{X_i}+N_i'+P_i') \cdot C \geq -2 \dim X .
		\end{equation}
		
		Assume that $ (K_{X_i}+B_i+M_i) \cdot C > 0 $. Since 
		$$ (K_{X_i}+B_i+M_i)\cdot C = \sum_{j=1}^m r_j \big(K_{X_i} + B_i^{(j)} + M_i^{(j)} \big)\cdot C $$
		by \eqref{eq:1_Bir11_g}, we infer that at least one number $ \big(K_{X_i} + B_i^{(j)} + M_i^{(j)} \big)\cdot C \in\frac{1}{r}\Z $ is non-zero. Thus,
		\begin{equation}\label{eq:9c}
			(K_{X_i}+B_i+M_i) \cdot C \geq \alpha .	
		\end{equation}
		Therefore, by \eqref{eq:9b}, \eqref{eq:2_Bir11_g}, \eqref{eq:9a} and \eqref{eq:9c} we obtain
		\begin{align*}
			0 &> (K_{X_i}+N_i+P_i) \cdot C \\
				&= (1-t)(K_{X_i}+B_i+M_i) \cdot C + t(K_{X_i}+N_i'+P_i') \cdot C \\
				&\geq (1-t) \alpha - 2t \dim X > 0 ,
		\end{align*}
		a contradiction. Consequently, $ (K_{X_i}+B_i+M_i) \cdot C = 0 $, which finishes the proof.
	\end{proof}

	\subsection{Lifting a sequence of flips with scaling}
	
	Let $ \big( X_1/Z, (B_1+N_1) + (M_1+P_1) \big) $ be a $\Q$-factorial NQC log canonical g-pair such that $ K_{X_1} + B_1 + N_1 + M_1 + P_1 $ is nef over $ Z $. Then we may run a $ (K_{X_1} + B_1 + M_1) $-MMP with scaling of $ P_1+N_1 $ over $Z$, and we assume now (for our purposes here) that it consists only of flips. Thus, we obtain the following diagram:
	
	\begin{center}
		\begin{tikzcd}[column sep = 0.8em, row sep = 2em]
			(X_1,B_1+M_1) \arrow[dr, "\theta_1" swap] \arrow[rr, dashed, "\pi_1"] && (X_2,B_2+M_2) \arrow[dl, "\theta_1^+"] \arrow[dr, "\theta_2" swap] \arrow[rr, dashed, "\pi_2"] && (X_3,B_3+M_3) \arrow[dl, "\theta_2^+"] \arrow[rr, dashed, "\pi_3"] && \cdots \\
			& Z_1 && Z_2
		\end{tikzcd}
	\end{center}
	We denote by $ P_i $ and $ N_i $ the pushforwards of $P_1$ and $ N_1 $ on $ X_i $, respectively, and by $ \lambda_i $ the corresponding \emph{nef thresholds}:
	\[ \lambda_i := \inf \{ t \in \R_{\geq 0} \mid K_{X_i} + (B_i + t N_i) + (M_i + t P_i) \text{ is nef over } Z \} . \]
	
	Let $ h_1 \colon (X_1',B_1'+M_1') \to (X_1,B_1+M_1) $ be a dlt blowup of $ (X_1,B_1+M_1) $. In particular, we have 
	\begin{equation}\label{eq:1_lifting}
		K_{X_1'} + B_1'+M_1' \sim_\R h_1^*(K_{X_1} + B_1 + M_1) . 
	\end{equation}
	Let $f_1\colon W\to X_1$ be a log resolution of all the divisors in sight on $X_1$ which dominates $X_1'$ and such that $P_1=(f_1)_*P_W$, where $P_W$ is a divisor on $W$ which is NQC over $ Z $. We may write 
	\[ f_1^*(P_1+N_1) = P_W + (f_1)_*^{-1} N_1 + E_1 , \]
	where $ E_1 $ is an $ f_1 $-exceptional $ \R $-divisor on $ W $ which is actually effective by the Negativity lemma and since $ N_1 \geq 0 $. We define $P_1'$ and $N_1'$ as the pushforwards of $P_W$ and $(f_1)_*^{-1} N_1+E_1$, respectively, to $X_1'$, and we note that
	\begin{equation}\label{eq:2_lifting}
		P_1' + N_1' = h_1^* (P_1 + N_1) . 
	\end{equation}
	
	\begin{lem}\label{lem:lifting_scaling_g}
		With notation and assumptions as above, there exists a diagram
		\begin{center}
			\begin{tikzcd}[column sep = 2em, row sep = large]
				(X_1',B_1'+M_1') \arrow[d, "h_1" swap] \arrow[rr, dashed, "\rho_1"] && (X_2',B_2'+M_2') \arrow[d, "h_2" swap]
				\\ 
				(X_1,B_1+M_1) \arrow[dr, "\theta_1" swap] \arrow[rr, dashed, "\pi_1"] && (X_2,B_2+M_2) \arrow[dl, "\theta_1^+"] 
				\\
				& Z_1
			\end{tikzcd}
		\end{center}
		where the map $ \rho_1 \colon X_1' \dashrightarrow X_2' $ is a $(K_{X_1'}+B_1'+M_1')$-MMP with scaling of $ P_1' + N_1' $ over $ Z_1 $ and the map $ h_2 \colon (X_2',B_2'+M_2') \to (X_2,B_2+M_2) $ is a dlt blowup of the g-pair $ (X_2,B_2+M_2) $. Moreover, this MMP is also a $(K_{X_1'}+B_1'+M_1')$-MMP with scaling of $ P_1' + N_1' $ over $ Z $, and we have
		$$ P_2' + N_2' = h_2^* (P_2 + N_2) ,$$
		where $ P_2' :=(\rho_1)_*P_1' $ and $N_2' :=(\rho_1)_*N_1'$.
	\end{lem}
	
	\begin{proof}
		Since $ (X_2,B_2+M_2) $ is the canonical model of $ (X_1,B_1+M_1) $ over $ Z_1 $ by definition, we infer that $ (X_2,B_2+M_2) $ is a minimal model of $ (X_1,B_1+M_1) $ over $ Z_1 $, and thus a dlt blowup of $ (X_2,B_2+M_2) $ is a minimal model in the sense of Birkar-Shokurov of $ (X_1,B_1+M_1) $ over $ Z_1 $. By \eqref{eq:1_lifting} and by \cite[Theorem 3.14]{HaconLiu21} we conclude that $ (X_1',B_1'+M_1') $ has a minimal model in the sense of Birkar-Shokurov of $ Z_1 $. Therefore, by Lemma \ref{lem:HL18_term_with_ample_scaling} there exists a $ (K_{X_1'} + B_1'+M_1') $-MMP with scaling of an ample divisor over $Z_1$ which terminates with a minimal model $ (X_2',B_2'+M_2') $ of $ (X_1',B_1'+M_1') $ over $Z_1$; we denote by $ \rho_1 \colon X_1' \dashrightarrow X_2' $ the induced map. Since $(X_2,B_2+M_2)$ is the canonical model of $(X_1',B_1'+M_1')$ over $ Z_1 $, by \cite[Lemma 2.12]{LMT} there exists a morphism $ h_2 \colon X_2' \to X_2 $ such that 
		\begin{equation}\label{eq:77}
			K_{X_2'} + B_2'+M_2' \sim_\R h_2^*(K_{X_2} + B_2 + M_2) .		
		\end{equation}
		In particular, the obtained g-pair $ (X_2',B_2'+M_2') $ is a dlt blowup of the g-pair $ (X_2,B_2+M_2) $.
		
		Now, by \eqref{eq:1_lifting} and \eqref{eq:2_lifting} we obtain 
		\begin{align}\label{eq:3_lifting}
			K_{X_1'} &+ (B_1' + \lambda_1 N_1') + (M_1' + \lambda_1 P_1') \\
			& \sim_\R h_1^* \big(  K_{X_1} + (B_1 + \lambda_1 N_1) + (M_1 + \lambda_1 P_1) \big) .\notag
		\end{align}
		Since $ K_{X_1} + (B_1 + \lambda_1 N_1) + (M_1 + \lambda_1 P_1) \equiv_{Z_1} 0 $ by construction of the $ (K_{X_1} + B_1 + M_1) $-MMP with scaling of $ P_1+N_1 $ over $Z_1$, by \eqref{eq:3_lifting} we infer that 
		\begin{equation}\label{eq:4_lifting}
			K_{X_1'} + (B_1' + \lambda_1 N_1') + (M_1' + \lambda_1 P_1') \equiv_{Z_1} 0 .
		\end{equation}
		Denote by $ Y^j \dashrightarrow Y^{j+1} $ the steps of this MMP, where $ Y^1 := X_1' $ and $ Y^k := X_2' $, and by $B^j$, $M^j$, $N^j$ and $P^j$ the pushforwards of $B_1'$, $M_1'$, $N_1'$ and $P_1'$, respectively, on $Y^j$. For each $ j \in \{1,\dots,k \} $ consider
		$$ \nu_j := \inf \big\{ t \in \R_{\geq 0} \mid K_{Y^j} + (B^j + t N^j) + (M^j + t P^j) \text{ is nef over } Z \big\} . $$
		Then by using \eqref{eq:3_lifting} and \eqref{eq:4_lifting} 
		 we can readily check that $ \nu_j = \lambda_1 $ for every $ j \in \{1,\dots,k-1 \} $.
		Therefore, the above $ (K_{X_1'} + B_1'+M_1') $-MMP with scaling of an ample divisor over $Z_1$ is automatically also a $ (K_{X_1'} + B_1'+M_1') $-MMP with scaling of $ N_1' + P_1' $ over $ Z $.
		
		If we set $ N_2' := N^k $ and $ P_2' := P^k $, then it follows from \eqref{eq:77} and \eqref{eq:3_lifting} that $ N_2' + P_2' = h_2^* (P_2 + N_2) $. This completes the proof.
	\end{proof}
	
	Therefore, by continuing this process analogously, we obtain:
	
	\begin{thm}
		\label{thm:lifting_scaling_g}
		Let $ \big( X_1/Z, (B_1+N_1) + (M_1+P_1) \big) $ be a $\Q$-factorial NQC log canonical g-pair such that $ K_{X_1} + B_1 + N_1 + M_1 + P_1 $ is nef over $ Z $. Assume that we have a sequence of flips for $(X_1,B_1+M_1)$ with scaling of $ P_1+N_1 $ over $Z$:
		\begin{center}
			\begin{tikzcd}[column sep = 0.8em, row sep = 2em]
				(X_1,B_1+M_1) \arrow[dr, "\theta_1" swap] \arrow[rr, dashed, "\pi_1"] && (X_2,B_2+M_2) \arrow[dl, "\theta_1^+"] \arrow[dr, "\theta_2" swap] \arrow[rr, dashed, "\pi_2"] && (X_3,B_3+M_3) \arrow[dl, "\theta_2^+"] \arrow[rr, dashed, "\pi_3"] && \cdots \\
				& Z_1 && Z_2
			\end{tikzcd}
		\end{center}
		Then there exists a diagram
		\begin{center}
			\begin{tikzcd}[column sep = 0.8em, row sep = large]
				(X_1',B_1'+M_1') \arrow[d, "h_1" swap] \arrow[rr, dashed, "\rho_1"] && (X_2',B_2'+M_2') \arrow[d, "h_2" swap] \arrow[rr, dashed, "\rho_2"] && (X_3',B_3'+M_3') \arrow[d, "h_3" swap] \arrow[rr, dashed, "\rho_3"] && \dots 
				\\ 
				(X_1,B_1+M_1) \arrow[dr, "\theta_1" swap] \arrow[rr, dashed, "\pi_1"] && (X_2,B_2+M_2) \arrow[dl, "\theta_1^+"] \arrow[dr, "\theta_2" swap] \arrow[rr, dashed, "\pi_2"] && (X_3,B_3+M_3) \arrow[dl, "\theta_2^+"] \arrow[rr, dashed, "\pi_3"] && \dots \\
				& Z_1 && Z_2
			\end{tikzcd}
		\end{center}
		where, for each $i \geq 1$, 
		\begin{enumerate}[\normalfont (a)]
			\item the map $ \rho_i \colon X_i' \dashrightarrow X_{i+1}' $ is a $(K_{X_i'}+B_i'+M_i')$-MMP with scaling of $ N_i' + P_i' $ over $ Z $, where the divisors $ N_i' $ and $ P_i' $ on $ X_i' $ are defined as in Lemma \ref{lem:lifting_scaling_g}, and
			
			\item the map $ h_i \colon (X_i',B_i'+M_i') \to (X_i,B_i+M_i) $ is a dlt blowup.
		\end{enumerate}
		In particular, the sequence on top of the above diagram is a $ (K_{X_1'}+B_1'+M_1') $-MMP with scaling of $ N_1' + P_1' $ over $ Z $.
	\end{thm}
	
	The same construction as above appears in \cite[Section 3.5]{HanLi18}, where the underlying variety $ X_1 $ of the g-pair $ (X_1,B_1+M_1) $ is additionally assumed to be klt. As demonstrated above, this assumption can be removed due to \cite{HaconLiu21}. For similar constructions see \cite[Section 3]{LMT} and \cite[Lemma 2.16]{CT20}.

	\subsection{Nakayama-Zariski decomposition and the MMP with scaling}
	
	Given a pseudo-effective $\R$-divisor $D$ on a smooth projective variety $X$, Nakayama \cite{Nak04} defined a decomposition $ D = P_\sigma (D) + N_\sigma (D) $, known as \emph{the Nakayama-Zariski decomposition of $D$}. This decomposition can be extended both to the singular setting, see for instance \cite[Subsection 2.2]{Hu20}, and to the relative setting, see \cite[Subsection III.4]{Nak04} and \cite[Section 2]{Les16}.
	
	We use here the relative Nakayama-Zariski decomposition of a pseudoeffective $ \R $-divisor on a normal $\Q$-factorial variety. Note that by \cite{Les16} this decomposition is not always well-defined. However, in all the cases we consider in this paper, the decomposition exists and behaves as in the absolute case. Below we recall briefly the basic definitions.
	
	\begin{dfn}
		Let $ \pi \colon X \to S $ be a projective surjective morphism between normal varieties. Assume that $ X $ is $ \Q $-factorial and fix a prime divisor $ \Gamma $ on $ X $. If $ D $ is a $ \pi $-big $ \R $-divisor on $ X $, then set
		\[ \sigma_\Gamma (D ; X/S) := \inf \{ \mult_\Gamma \Delta \mid \Delta \geq 0 \text{ and } \Delta \equiv_S D \} . \]
		If $ D $ is a $ \pi $-pseudoeffective $ \R $-divisor on $ X $, then pick a $ \pi $-ample $ \R $-divisor $ A $ on $ X $ and set 
		\[ \sigma_\Gamma (D ; X/S) := \lim\limits_{\varepsilon \to 0^+} \sigma_\Gamma (D + \varepsilon A ; X/S) . \]
		Note that $ \sigma_\Gamma (D ; X/S) $ does not depend on the choice of $ A $. Set 
		\[ N_\sigma (D ; X/S) := \sum_\Gamma \sigma_\Gamma (D ; X/S) \cdot \Gamma , \]
		where this formal sum runs through all prime divisors $\Gamma$ on $X$.
	\end{dfn}
	
	We will need the following lemma in the proof of Theorem \ref{thm:HL18_term_with_scaling}.
	
	\begin{lem}\label{lem:druel}
		Let $\big(X/Z,(B+N)+(M+P)\big)$ be a $\Q$-factorial NQC log canonical g-pair such that $ K_X + B + M $ is pseudoeffective over $Z$ and $ K_X + B + N + M + P $ is nef over $Z$. Assume that $N_\sigma(K_X+B+M; X/Z)$ is an $\R$-divisor.\footnote{This assumption is satisfied if, for instance, $(X,B+M)$ has a minimal model over $Z$: this follows from the Negativity lemma and from the relative versions of \cite[Lemmas III.5.14 and III.5.15]{Nak04} as in the proof of \cite[Lemma 2.4(1)]{HX13}.} Consider a $ (K_X + B + M) $-MMP with scaling of $ P+N $ over $Z$ and denote by $ \lambda_i $ the corresponding nef thresholds in the steps of this MMP. If $ \lim\limits_{i\to \infty} \lambda_i = 0$, then this MMP contracts precisely the components of $N_\sigma(K_X+B+M; X/Z)$.
	\end{lem}

	\begin{proof}
		Let $X_i$ be the steps of the $(K_X+B+M)$-MMP with scaling of $N+P$ over $Z$ and let $B_i$, $M_i$, $N_i$ and $P_i$ be the pushforwards on $X_i$ of $B$, $M$, $N$ and $P$, respectively, where $X=X_1$. Since the map $X\dashrightarrow X_i$ is a partial MMP for every $i$, it can contract only components of $N_\sigma(K_X+B+M; X/Z)$, cf.\ \cite[Lemma 2.4(1)]{HX13}.
	
		On the other hand, since the divisor $K_{X_i}+B_i+M_i+\lambda_i(N_i+P_i)$ is nef over $ Z $, by the Negativity lemma the map $X\dashrightarrow X_i$ contracts precisely the components of $N_\sigma\big(K_X+B+M+\lambda_i(N+P); X/Z\big)$. Since
		$$\Supp N_\sigma(K_X+B+M; X/Z)\subseteq\bigcup_{i\in\N}\Supp N_\sigma\big(K_X+B+M+\lambda_i(N+P); X/Z\big)$$
		by the relative version of \cite[Lemma III.1.7(2)]{Nak04}, every component of $N_\sigma(K_X+B+M; X/Z)$ must be contracted by some map $X\dashrightarrow X_i$. This completes the proof.
	\end{proof}

	\section{Weak Zariski decompositions}
	
	We prove a variant of \cite[Theorem 3.2]{LT22}, which plays a crucial role in this paper and is also of independent interest. Even though the strategy of the proof is the same, there are several small technical issues and we provide all the details for the benefit of the reader.
	
	\begin{thm}\label{thm:WZD_existence_g}
		Assume the existence of minimal models for smooth varieties of dimension $n-1$.
		
		Let $\big(X/Z,(B+N)+(M+P)\big)$ be an NQC $\Q$-factorial dlt g-pair of dimension $n$. Assume that the divisor $ K_X+B+N+M+P $ is pseudoeffective over $ Z $ and that for each $\varepsilon>0$ the divisor $K_X+B+M+(1-\varepsilon)(N+P)$ is not pseudoeffective over $Z$. Then the g-pair $\big(X,(B+N)+(M+P)\big)$ admits an NQC weak Zariski decomposition over $Z$.
	\end{thm}
	
	\begin{proof}
		We proceed in four steps.
		
		\medskip
		
		\emph{Step 1.} In this step we show that we may assume the following:
		
		\medskip
		
		\emph{Assumption 1}.
		There exists a fibration $\xi \colon X \to Y$ over $Z$ to a normal quasi-projective variety $Y$ such that $\dim Y<\dim X$ and such that:
		\begin{enumerate}
			\item[(a$_1$)] $\nu\big(F,(K_X+B+N+M+P)|_F\big)=0$ and $h^1(F,\OO_F)=0$ for a very general fibre $F$ of $\xi$, 
			
			\item[(b$_1$)] $K_X+B+M+(1-\varepsilon) (N+P)$ is not $\xi$-pseudoeffective for any $\varepsilon>0$. 
		\end{enumerate}
		
		\medskip
		
		To this end, by \cite[Lemma 4.3]{HanLiu20} and its proof there exist a birational contraction $ \varphi \colon X \dashrightarrow S$ over $Z$ and a fibration $f \colon S \to Y$ over $Z$ such that:
		\begin{enumerate}
			\item[(a)] $ \big(S,(B_S+N_S)+(M_S+P_S)\big) $ is an NQC $\Q$-factorial log canonical g-pair, where $B_S$, $N_S$, $M_S$ and $P_S$ are pushforwards of $ B $, $N$, $M$ and $P$ on $S$, respectively,
			
			\item[(b)] $ Y $ is a normal quasi-projective variety with $\dim Y<\dim X$,
			
			\item[(c)] $ K_S + B_S + N_S + M_S + P_S \sim_{\R,Y} 0 $,
			
			\item[(d)] $\varphi$ is a $\big(K_X+B+M+(1-\zeta)(N+P)\big)$-MMP for some $0 < \zeta \ll 1$ and $f$ is the corresponding Mori fibre space,
			
			\item[(e)] $ S $ has klt singularities.
		\end{enumerate}
		
		Let $(p,q) \colon W \to X \times S$ be a resolution of indeterminacies of $\varphi$ such that $ W $ is a log resolution of all the divisors in sight on $X$ and there exist nef $ \R $-divisors $M_W$ and $P_W$ on $W$ such that $M=p_*M_W$ and $P=p_*P_W$. If $ B_W $ and $ N_W $ are the strict transforms of $ B $ and $ N $ on $ W $, respectively, then we may write
		\begin{equation}\label{eq:1_WZD_existence_g}
			K_W+B_W+N_W+M_W+P_W+G_W \sim_\R p^*(K_X+B+N+M+P)+E_W,
		\end{equation}
		where the divisors $G_W$ and $E_W$ are effective, $ p $-exceptional and have no common components. 		
		\begin{center}
			\begin{tikzcd}[column sep = large]
				& W \arrow[dl, "p" swap] \arrow[dr, "q"] && \\
				X \arrow[rr, dashed, "\varphi" ] && S \arrow[r, "f"] & Y 
			\end{tikzcd}
		\end{center}
	
		Let $F$ be a very general fibre of $f$ and set $F_W := q^{-1}(F) \subseteq W$. In addition, set 
		\[ N_F:= N_S|_F , \quad B_F := B_S|_F , \quad M_F := M_S |_F , \quad P_F := P_S |_F \]
		and similarly for $N_{F_W}$, $B_{F_W}$, $M_{F_W}$, $P_{F_W}$ and $G_{F_W}$. Then both the divisors $K_F+B_F+N_F+M_F+P_F$ and $K_{F_W}+B_{F_W}+N_{F_W}+M_{F_W}+P_{F_W}+G_{F_W}$ are pseudoeffective and we have 
		\[ (q|_{F_W})_*(K_{F_W}+B_{F_W}+N_{F_W}+M_{F_W}+P_{F_W}+G_{F_W}) = K_F+B_F+N_F+M_F+P_F \]
		by \eqref{eq:1_WZD_existence_g}. By \cite[Lemma 2.8]{LP18} and by (c), we obtain
		\begin{align*}
			\nu(F_W,K_{F_W}&+B_{F_W}+N_{F_W}+M_{F_W}+P_{F_W}+G_{F_W}) \\
			& \leq \nu(F,K_F+B_F+N_F+M_F+P_F)=0 ,
		\end{align*}
		hence $\nu(F_W,K_{F_W}+B_{F_W}+N_{F_W}+M_{F_W}+P_{F_W}+G_{F_W})=0$.
		
		For every $\varepsilon>0$, the divisor $K_{F_W}+B_{F_W}+M_{F_W}+(1-\varepsilon)(N_{F_W}+P_{F_W})+G_{F_W}$ is not pseudoeffective, since otherwise the divisor 
		\begin{align*}
		K_F&+B_F+M_F+(1-\varepsilon)(N_F+P_F)\\
		& =(q_{F_W})_*\big(K_{F_W}+B_{F_W}+M_{F_W}+(1-\varepsilon)(N_{F_W}+P_{F_W})+G_{F_W}\big)
		\end{align*}
		would be pseudoeffective for some $\varepsilon>0$, a contradiction to (c) and (d).
		
		Since $ S $ has klt singularities by (e), so does $F$, and hence $F$ has rational singularities. Additionally, $h^1(F,\OO_F)=0$ by (d) and by the Kodaira vanishing theorem. It follows that $h^1(F_W,\OO_{F_W})=0$.
		
		If $K_W+B_W+N_W+M_W+P_W+G_W$ admits an NQC weak Zariski decomposition over $ Z $, then $K_X+B+N+M+P$ admits an NQC weak Zariski decomposition over $ Z $ by \eqref{eq:1_WZD_existence_g} and by \cite[Lemma 2.14]{LT22}.
		
		In conclusion, by replacing $ \big(X,(B+N)+(M+P)\big) $ with $\big(W,(B_W+N_W+G_W)+(M_W+P_W)\big)$ and by setting $\xi: = f \circ q$, we achieve Assumption 1.
		
		\medskip
		
		\emph{Step 2.} If $\dim Y=0$ (and thus necessarily $\dim Z=0$), then 
		\[K_X+B+N+M+P \equiv N_\sigma(K_X+B+N+M+P)\]
		by \cite[Proposition V.2.7(8)]{Nak04} and by (a$_1$). Hence, $K_X+B+N+M+P$ admits an NQC weak Zariski decomposition, and we are done.
		
		\medskip
				
		\emph{Step 3.} Assume from now on that $\dim Y>0$.
		In this step we show that we may assume the following:
		
		\medskip
		
		\emph{Assumption 2}.
		There exists a fibration $g \colon X \to T$ to a normal quasi-projective variety $T$ over $Z$ such that:
		\begin{enumerate}
			\item[(a$_2$)] $0<\dim T<\dim X$,
			
			\item[(b$_2$)] $K_X+B+N+M+P\sim_{\R,T} 0$.
			
%
%
		\end{enumerate} 
		\emph{However, instead of the g-pair $\big(X,(B+N)+(M+P)\big)$ being $\Q$-factorial dlt, we may only assume that it is an NQC log canonical g-pair such that $(X,0)$ is $\Q$-factorial klt.}
		
		\medskip
		
		Note that $\big(X,(B+N)+(M+P)\big)$ is also a g-pair over $Y$. It follows from (a$_1$) and \cite[Corollary 2.18]{LT22} that the divisor $K_X+B+N+M+P$ is effective over $Y$. Hence, by \cite[Theorem 4.4(ii)]{LT22} we may run a $(K_X+B+N+M+P)$-MMP with scaling of an ample divisor over $Y$ which terminates. We obtain thus a birational contraction $ \theta \colon X \dashrightarrow X' $ over $ Y $ and a g-pair $ \big(X',(B'+N')+(M'+P')\big) $ such that the divisor $ K_{X'} + B' + N' + M' + P' $ is nef over $ Y $, where $B'$, $ N'$, $M'$ and $P'$ are pushforwards of $B$, $N$, $M$ and $P$ on $X'$, and we denote by $\xi' \colon X' \to Y$ the induced morphism.
		
		By Proposition \ref{prop:Bir11_prop_g} there exists $\delta>0$ such that, if we run a $ \big(K_{X'}+B'+M'+(1-\delta)(N'+P') \big) $-MMP with scaling of an ample divisor over $Y$, then this MMP is $(K_{X'}+B'+N'+M'+P')$-trivial. In addition, note that $ K_{X'}+B'+M'+(1-\delta)(N'+P')$ is not $\xi'$-pseudoeffective: indeed, by possibly choosing $\delta$ smaller, we may assume that the map $\theta$ is $ \big(K_{X'}+B'+M'+(1-\delta)(N'+P')\big)$-negative, and the claim follows since $K_X+B+M+(1-\delta)(N+P)$ is not $\xi$-pseudoeffective by (b$_1$). Therefore, by \cite[Lemma 4.4(1)]{BZ16} this relative $ \big(K_{X'}+B'+M'+(1-\delta)(N'+P') \big) $-MMP terminates with a Mori fibre space $f'' \colon X'' \to Y''$ over $Y$. We obtain a birational contraction $\theta' \colon X' \dashrightarrow X''$ and a g-pair $ \big(X'',(B''+N'')+(M''+P'')\big) $, where $B''$, $ N''$, $M''$ and $P''$ are the appropriate pushforwards on $X''$, and we denote by $\xi'' \colon X'' \to Y$ the induced morphism.
		
		\begin{center}
			\begin{tikzcd}[column sep = large, row sep = 2.5em]
				X \arrow[dr, "\xi" swap] \arrow[r, dashed, "\theta"] & X' \arrow[r, dashed, "\theta'"] \arrow[d, "\xi'"] & X'' \arrow[d, "f''"] \arrow[dl] \\
				& Y & Y'' \arrow[l, "\xi''"] 
			\end{tikzcd}
		\end{center}
		
		Then the variety $ X'' $ is $ \Q $-factorial, the g-pair $ \big(X'',(B''+N'')+(M''+P'')\big) $ is NQC log canonical, the pair $(X'',0)$ is klt by Lemma \ref{lem:sing_smaller_boundary} since the g-pair $ \big(X'',B''+M''+(1-\delta)(N''+P'')\big) $ is dlt, and by Proposition \ref{prop:Bir11_prop_g} we have
		\[ K_{X''} + B'' + N'' + M'' + P'' \equiv_{Y''} 0. \]
		Furthermore, the numerical equivalence over $Y''$ coincides with the $\R$-linear equivalence over $Y''$, since $f''$ is an extremal contraction, see \cite[Theorem 1.3(4)(c)]{HaconLiu21}.
		
		Since the composite map $\theta' \circ \theta \colon X \dashrightarrow X'' $ is $ (K_X+B+N+M+P)$-non\-positive by construction, it follows from \cite[Lemma 2.14]{LT22} that if $K_{X''} + B'' + N'' + M'' + P''$ admits an NQC weak Zariski decomposition over $ Z $, then so does $K_X+B+N+M+P$.
		
		In conclusion, by replacing $ \big(X,(B+N)+(M+P)\big) $ with $\big(X'',(B''+N'')+(M''+P'')\big)$ and by setting $T:=Y''$ and $ g := f'' $, we achieve Assumption 2.
		
		\medskip
		
		\emph{Step 4.} By (b$_2$) and by \cite[Theorem 1.2]{HanLiu21} there exists an NQC log canonical g-pair $ (T/Z, B_T + M_T) $ such that
		\begin{equation}\label{eq:3_WZD_existence_g}
			K_X + B + N + M + P \sim_\R g^* (K_T + B_T + M_T).
		\end{equation}
		The divisor $ K_T + B_T + M_T $ is pseudoeffective over $ Z $ by \eqref{eq:3_WZD_existence_g}. Therefore, by (a$_2$) and by the assumptions of the theorem in lower dimensions together with Remark \ref{rem:MMimplWZD} and \cite[Theorem E]{LT22}, the g-pair $ (T/Z, B_T + M_T) $ admits an NQC weak Zariski decomposition over $Z$. Then \eqref{eq:3_WZD_existence_g} and \cite[Remark 2.11]{LT22} imply that the g-pair $\big(X/Z,(B+N)+(M+P)\big)$ admits an NQC weak Zariski decomposition over $Z$, as desired.
	\end{proof}
		
	The following consequence of Theorem \ref{thm:WZD_existence_g} plays an important role in the proofs of Proposition \ref{prop:HH19_prop_g} and Theorem \ref{thm:HH19_thm_g}.
	
	\begin{lem}\label{lem:MM_existence_HH19}
		Assume the existence of minimal models for smooth varieties of dimension $ n-1 $.
		
		Let $\big(X/Z,(B+N)+(M+P)\big)$ be an NQC $ \Q $-factorial log canonical g-pair of dimension $n$ such that $ K_X+B+N+M+P$ is pseudoeffective over $ Z $. Assume that the g-pair $ (X,B+M) $ admits an NQC weak Zariski decomposition over $ Z $ or that $ K_X+B+M $ is not pseudoeffective over $ Z $. Then $ \big(X/Z,(B+N)+(M+P)\big) $ has a minimal model in the sense of Birkar-Shokurov over $ Z $.
	\end{lem}
	
	\begin{proof}
		We consider the two cases separately.
		
		First, if the g-pair $ (X,B + M) $ admits an NQC weak Zariski decomposition over $ Z $, then the conclusion follows from Lemma \ref{lem:MM_bigger_boundary_g}.
		
		Assume now that $ K_X + B + M $ is not pseudoeffective over $ Z $. Set
		$$ \tau := \inf \big\{ t \in \R_{\geq 0} \mid K_X+B+M+t(N+P) \text{ is pseudoeffective over } Z \big\} , $$
		and observe that $ \tau \in (0,1] $. Take a dlt blowup 
		\[ h \colon \big(X', (B' + \tau h_*^{-1} N) +(M'+\tau P')\big) \to \big(X,(B+\tau N)+(M+\tau P) \big) \]
		of the g-pair $ \big(X,(B+\tau N)+(M+\tau P) \big) $, where $ B' := h_*^{-1} B + E $ and $ E $ is the sum of all $ h $-exceptional prime divisors. Then the g-pair $ \big(X', B' + M' \big) $ is dlt by Lemma \ref{lem:sing_smaller_boundary} and the divisor $ K_{X'} + B' + M' $ is not pseudoeffective over $ Z $, since the divisor $ K_X + B + M = h_*(K_{X'} + B' + M')$ is not pseudoeffective over $ Z $. Additionally, since 
		\[ K_{X'} + B' + \tau h_*^{-1} N + M'+\tau P' \sim_\R  h^* \big( K_X + B + \tau N + M +\tau P \big) \]
		is pseudoeffective over $ Z $, we infer similarly that
		\[ \tau = \inf \{ t \in \R_{\geq 0} \mid K_{X'} + B' + M' + t (h_*^{-1} N+P') \text{ is pseudoeffective over } Z \}. \]
		By Theorem \ref{thm:WZD_existence_g} the g-pair $ \big(X', (B' + \tau h_*^{-1} N) +(M'+\tau P')\big) $ admits an NQC weak Zariski decomposition over $ Z $, hence so does the g-pair $\big(X,(B+\tau N)+(M+\tau P) \big)$ by \cite[Remark 2.11]{LT22}. We conclude by Lemma \ref{lem:MM_bigger_boundary_g}.
	\end{proof}

	\section{On termination of the MMP with scaling}
	\label{section:Term_with_Scaling}
	
	The following result improves on \cite[Theorem 4.1]{HanLi18} by removing the assumption that the underlying variety is klt. Its proof follows the same strategy as the proofs of \cite[Theorem 4.1(iii)]{Bir12a} and \cite[Theorem 4.1]{HanLi18}. The organisation of the proof is somewhat different and occasionally streamlined. We make an extra effort to provide all the details, since the arguments are quite involved.
	
	\begin{thm}\label{thm:HL18_term_with_scaling}
		Let $ (X/Z,B+M) $ be a $\Q$-factorial NQC g-pair, let $ P $ be the pushforward of an NQC divisor on a high birational model of $ X $, and let $ N $ be an effective $ \R $-divisor on $ X $. Assume that the g-pair $ \big( X, (B+N) + (M+P) \big) $ is log canonical and that the divisor $ K_X + B + N + M + P $ is nef over $ Z $. Consider a $ (K_X + B + M) $-MMP over $Z$ with scaling of $ P+N $, denote by $ \lambda_i $ the corresponding nef thresholds in the steps of this MMP and set $ \lambda := \lim\limits_{i\to \infty} \lambda_i $.
		
		If $ \lambda \neq \lambda_i $ for every $ i $ and if $\big(X,(B+\lambda N)+(M+\lambda P)\big)$ has a minimal model in the sense of Birkar-Shokurov over $Z$, then the given MMP terminates.
	\end{thm}
	
	\begin{proof}
		We proceed in several steps.
			
		\medskip
				
		\emph{Step 1.}
		First, by replacing $(X,B+M)$ with $\big(X,B+M+\lambda (P+N)\big)$, we may assume that $\lambda=0$. We may also assume that the MMP consists only of flips. Next, by arguing as in the first paragraphs of the proof of \cite[Theorem 5.1]{LMT} we may further assume that each birational map $ X_i \dashrightarrow X_{i+1} $ in the MMP is an isomorphism at the generic point of each log canonical centre of $(X_i,B_i+M_i)$. Pick an index $i$ such that $\lambda_i>\lambda_{i+1}$ and note that the g-pair $\big(X_{i+1}, B_{i+1}+M_{i+1}+\lambda_i(N_{i+1}+P_{i+1})\big)$ is log canonical. Therefore, by replacing $(X,B+M)$ with $(X_{i+1},B_{i+1}+M_{i+1})$ and $N+P$ with $\lambda_{i+1}(N_{i+1}+P_{i+1})$, we may assume that
		$$\lambda_1=1$$
		and
		\begin{equation}\label{eq:5}
			\big(X,B+M+(1+\varepsilon) (N+P)\big)\text{ is log canonical for }0\leq\varepsilon\ll 1.
		\end{equation}

		\medskip
		
		\emph{Step 2.} 
		By assumption, there exists a birational map $\varphi\colon X \dashrightarrow Y$ to a minimal model in the sense of Birkar-Shokurov $(Y,B_Y+M_Y)$ of $(X,B+M)$ over $ Z $.  Let $(f,g)\colon W \to X\times Y$ be a resolution of indeterminacies of the map $\varphi$ which is a log resolution of all relevant divisors in sight and such that there exists divisors $M_W$ and $P_W$ on $W$ which are nef over $Z$ such that $f_*M_W=M$ and $f_*P_W=P$. Set $N_W:=f_*^{-1}N$ and let $B_W$ be the sum of $f_*^{-1}B$ and of all $f$-exceptional prime divisors on $W$. 
	
		\medskip
		
		\emph{Step 2a.} In this step we construct a dlt blowup of $(Y',B_{Y'}+M_{Y'})$ of $(Y, B_Y+M_Y)$.
		
		The divisor
		$$ F := f^*(K_{X}+B+M) -g^*(K_Y+B_Y+M_Y) $$
		is effective and $g$-exceptional and the divisor
		$$F' := K_W+B_W+M_W - f^*(K_{X}+B+M) $$
		is effective and both $f$-exceptional and $g$-exceptional by Remark \ref{rem:1}. We have
		$$ K_W+B_W+M_W \equiv_Y F+F'. $$
		By \cite[Proposition 3.8]{HanLi18} we can run a $(K_W+B_W+M_W)$-MMP over $Y$ with scaling of an ample divisor which terminates with a model $(Y',B_{Y'}+M_{Y'})$ on which the divisor $F+F'$ is contracted. We denote by $\theta \colon W \dashrightarrow Y'$ the induced birational contraction.
	
		\medskip
		
		\emph{Step 2b.} In this step we construct a dlt blowup of $(X',B'+M')$ of $(X,B+M)$ with certain additional properties that we will need later.
	
		Set
		$$ F'' := K_W+(B_W+N_W)+(M_W+P_W) -f^*\big(K_X+(B+N)+(M+P)\big).$$
		Since the g-pair $ \big( X, (B+N) + (M+P) \big) $ is log canonical, $F''$ is effective and $f$-exceptional. By \cite[Proposition 3.8]{HanLi18} we can run a $\big(K_W+(B_W+N_W)+(M_W+P_W)\big)$-MMP over $X$ with scaling of an ample divisor which terminates with a model $\big(X',(B'+N')+(M'+P')\big)$ on which the divisor $F''$ is contracted. We denote by $\xi \colon W \dashrightarrow X'$ the induced birational contraction and by $h \colon X' \to X$ the induced map. 

		\begin{center}
			\begin{tikzcd}[column sep = 2em, row sep = large]
				(X',B'+M') \arrow[rrd, "h"] && (W,B_W+M_W) \arrow[d, "f" swap] \arrow[ll, dashed, "\xi" swap] \arrow[rr, dashed, "\theta"] \arrow[rrd, "g"] && (Y',B_{Y'}+M_{Y'}) \arrow[d]
				\\ 
				&& (X,B+M) \arrow[rr, dashed, "\varphi"] && (Y,B_Y+M_Y) 
			\end{tikzcd}
		\end{center}
	
		Then 
		$$ K_{X'}+(B'+N')+(M'+P') \sim_\R h^*\big(K_X+(B+N)+(M+P)\big). $$
		Write
		$$ h^*(N+P)=N'+P'+\Theta'. $$
		Since $N'=h_*^{-1}N$, the divisor $h^*N-N'$ is effective and $h$-exceptional, and the divisor $h^*P-P'$ is effective and $h$-exceptional by the Negativity lemma. Therefore, $\Theta'$ is an effective and $h$-exceptional divisor, hence every component of $\Theta'$ is a log canonical place of $\big(X,(B+N)+(M+P)\big)$ on $X'$. This and \eqref{eq:5} imply that $\Theta'=0$, and consequently
		$$ h^*(N+P)=N'+P' $$
		and
		\[ K_{X'}+B'+M' \sim_\R h^*(K_X+B+M). \]
		The map $h$ is therefore a dlt blowup of $(X,B+M)$. Note that an $f$-exceptional prime divisor on $W$ is contracted by $\xi$ if and only if it is not a log canonical place of $\big(X,(B+N)+(M+P)\big)$ on $W$ by construction. Thus, an $f$-exceptional prime divisor on $W$ is contracted by $\xi$ if and only if it is not a log canonical place of $(X,B+M)$ on $W$ by \eqref{eq:5}.
		
		\medskip
		
		\emph{Step 3.} 
		In this step we show that $\varphi$ does not contract any divisor.
		
		Assume for contradiction that $D$ is a $ \varphi $-exceptional prime divisor on $ X $ and let $D_W:=f_*^{-1}D$. Then $a(D,X,B+M)<a(D,Y,B_Y+M_Y)$ by the definition of minimal models in the sense of Birkar-Shokurov, and thus $D_W\subseteq\Supp F$. Therefore, the map $\theta$ contracts $D_W$ by Step 2a.
		
		In the remainder of this step we will derive a contradiction by showing that $D_W$ cannot be contracted by $\theta$.
		
		Pick $0<\lambda_i\ll1$ such that the map $\theta$ is $\big(K_W+B_W+M_W+\lambda_i (N_W+P_W)\big)$-negative. Since $\big(X,(B+M)+\lambda_i(N+P)\big)$ is log canonical and by the definition of $B_W$, the divisor
		$$G:=K_W+B_W+M_W+\lambda_i (N_W+P_W)-f^{*}\big(K_X+B+M+\lambda_i(N+P)\big)$$
		is effective and $f$-exceptional.
		
		Next, let $(p,q)\colon W'\to X\times X_i$ be a resolution of indeterminacies of the birational contraction $X\dashrightarrow X_i$, which dominates $W$. Then
		\begin{equation}\label{eq:4}
				p^*\big(K_X+B+M+\lambda_i(N+P)\big)\sim_\R q^*(K_{X_{i}}+B_i+M_i+\lambda_i(N_i+P_i)\big)+E,	
		\end{equation}
		where $E$ is effective and $p$-exceptional, since the map $X\dashrightarrow X_i$ is an isomorphism in codimension $1$, see Step 1. We denote by $Q$ the pushforward of $q^*(K_{X_{i}}+B_i+M_i+\lambda_i(N_i+P_i)\big)$ to $W$. Then $Q$ is movable over $ Z $, since the divisor $K_{X_{i}}+B_i+M_i+\lambda_i(N_i+P_i)$ is nef over $ Z $. Pushing forward the relation \eqref{eq:4} to $W$, we obtain
		$$f^*\big(K_X+B+M+\lambda_i(N+P)\big)\sim_\R Q+E',$$ 
		where $E'\geq0$ is $f$-exceptional. This yields
		\begin{equation}\label{eq:90}
			K_W+B_W+M_W+\lambda_i (N_W+P_W)\sim_\R Q+G+E'. 
		\end{equation}
		Since the divisor $G+E'$ is $f$-exceptional,
		\begin{equation}\label{eq:90_1}
			\text{the prime divisor $D_W = f_*^{-1}D$ is not a component of $G+E'$.}
		\end{equation}

		Now, assume that $D_W$ is contracted at the $\ell$-th step $ \mu_\ell \colon W_\ell \dashrightarrow W_{\ell+1} $ of the MMP $\theta \colon W \dashrightarrow Y'$. In particular, $ \mu_\ell  $ is a divisorial contraction. Denote by $ B_{W_\ell} $, $ D_{W_\ell} $, $ E'_{W_\ell} $, $ G_{W_\ell} $, $ M_{W_\ell} $, $ N_{W_\ell} $, $ P_{W_\ell} $ and $ Q_{W_\ell} $ the pushforwards on $ W_\ell $ of the divisors $ B_W $, $ D_W $, $ E' $, $ G $, $ M_W $, $ N_W $, $ P_W $ and $ Q $, respectively. From \eqref{eq:90} we have
		\begin{equation}\label{eq:90_interm}
			K_{W_\ell}+B_{W_\ell}+M_{W_\ell}+\lambda_i (N_{W_\ell}+P_{W_\ell})\sim_\R Q_{W_\ell}+G_{W_\ell}+E'_{W_\ell} .
		\end{equation}
		Since $ D_{W_\ell} $ is not a component of $ G_{W_\ell} + E'_{W_\ell} $ by \eqref{eq:90_1}, for a general curve $C_\ell$ contracted by $\mu_\ell$ we would have $(G_{W_\ell} + E'_{W_\ell}) \cdot C_\ell\geq 0$, as well as $Q_{W_\ell} \cdot C_\ell \geq0$, since $ Q_{W_\ell} $ is movable over $ Z $. On the other hand, we have $\big( K_{W_\ell}+B_{W_\ell}+M_{W_\ell}+\lambda_i (N_{W_\ell}+P_{W_\ell}) \big) \cdot C_\ell <0$ by the choice of $\lambda_i$, which contradicts \eqref{eq:90}. This proves the claim.
		
		\medskip
	
		\emph{Step 4.} In this step we show that $(Y',B_{Y'}+M_{Y'})$ is a minimal model in the usual sense of $(X',B'+M')$ over $Z$.
		
		Assume first that $S$ is a prime divisor on $ Y' $ which is contracted to $X'$. Let $S_W$ be the strict transform of $S$ on $W$. If $S$ is not exceptional over $Y$, then $a(S_W,X,B+M)=-1$ by the definition of minimal models in the sense of Birkar-Shokurov. If $S$ is exceptional over $Y$, then $a(S_W,Y,B_Y+M_Y)=-1$ since $(Y',B_{Y'}+M_{Y'})$ is a dlt blowup of $(Y, B_Y+M_Y)$, hence we have again $a(S_W,X,B+M)=-1$ by Remark \ref{rem:1}. But this contradicts the last sentence in Step 2b. Therefore, the map $X' \dashrightarrow Y'$ is a birational contraction.
		
		Assume now that $S$ is a prime divisor on $X'$ which is contracted to $Y'$. Then $S$ is contracted to $Y$, and since $\varphi$ does not contract any divisor by Step 3, the divisor $S$ must be contracted to $X$. Since $(X',B'+M')$ is a dlt blowup of $(X,B+M)$, we must have
		\begin{equation}\label{eq:91}
			a(S, X', B'+M')=a(S, X, B+M)=-1.	
		\end{equation}
		Therefore, we need to show that $a(S,Y',B_{Y'}+M_{Y'})>-1$. 
		
		Assume on the contrary that $a(S,Y',B_{Y'}+M_{Y'})=-1$. Then we also have $a(S,Y,B_Y+M_Y)=-1$, since $(Y',B_{Y'}+M_{Y'})$ is a dlt blowup of $(Y,B_Y+M_Y)$. Recalling the definitions of the divisors $F$ and $F'$ from Step 2a, if $S_W$ is the strict transform of $S$ on $W$, then the previous sentence implies that $S_W\nsubseteq \Supp F$. Moreover, \eqref{eq:91} implies that $S_W\nsubseteq \Supp F'$. Hence, $S_W$ is not contracted by $\theta$ by Step 2a, a contradiction which shows the claim.
		
		\medskip
	
		\emph{Step 5.}
		By Theorem \ref{thm:lifting_scaling_g} we can lift the sequence $X_i\dashrightarrow X_{i+1}/Z_i$ to a $(K_{X'}+B'+M')$-MMP over $Z$ with scaling of $N'+P'$, where each g-pair $({X_i'}, B_i'+M_i')$ is $\Q$-factorial and dlt. 
	
		\begin{center}
			\begin{tikzcd}[column sep = 2em, row sep = large]
				(X',B'+M') \arrow[rrd, "h"] \arrow[d, dashed] && (W,B_W+M_W) \arrow[d, "f" swap] \arrow[ll, dashed, "\xi" swap] \arrow[rr, dashed, "\theta"] \arrow[rrd, "g"] && (Y',B_{Y'}+M_{Y'}) \arrow[d]
				\\ 
				(X_i', B_i'+M_i') \arrow[rrd] && (X,B+M) \arrow[rr, dashed, "\varphi"] \arrow[d, dashed] && (Y,B_Y+M_Y) 
				\\
				&& (X_i, B_i+M_i) && 
			\end{tikzcd}
		\end{center}
	
		By Step 4 and by Lemma \ref{lem:druel} the varieties $X_i'$ and $Y'$ are isomorphic in codimension $1$ for all $i\gg0$.
	
		\medskip
		
		\emph{Step 6.} To finish the proof of the theorem, it suffices to show that the MMP 
		\[ (X_1', B_1'+M_1') \dashrightarrow (X_2', B_2'+M_2') \dashrightarrow \dots \dashrightarrow (X_i', B_i'+M_i') \dashrightarrow \cdots \]
		terminates, where $(X_1', B_1'+M_1'):=(X', B'+M')$.
	
		Suppose that this MMP does not terminate. After relabelling we may assume that this MMP consists only of flips and that all $X_i'$ and $Y'$ are isomorphic in codimension $1$ by Step 5. 
			
		Let $A$ be a reduced effective divisor on $W$ whose components are divisors which are ample over $Z$ and whose classes generate $N^1(W/Z)_\R$. Since the map $\xi\colon W\dashrightarrow X'$ is obtained by running a $(K_W+B_W+M_W+N_W+P_W)$-MMP over $ X $ by Step 2b, there exists $0<\varepsilon_0\ll1$ such that this map is also $(K_W+B_W+M_W+N_W+P_W+\varepsilon A)$-negative for each $0<\varepsilon\leq\varepsilon_0$, so that the g-pair $\big(X'/Z,(B'+N'+\varepsilon A')+(M'+P')\big)$ is dlt, where $A'$ is the strict transform of $A$ on $X'$. 
		
		Similarly we may assume that $\big(Y'/Z,(B_{Y'}+\varepsilon A_{Y'})+M_{Y'})$ is dlt for each $0<\varepsilon\leq\varepsilon_0$, where $A_{Y'}$ is the strict transform of $A$ on $Y'$.
	
		Pick an index $i$ and $0<\varepsilon_0\ll1$ such that additionally:
		\begin{enumerate}
			\item[(a)] $\lambda_{i-1} >\lambda_i$,
			
			\item[(b)] any $\big(K_{Y'}+B_{Y'}+M_{Y'}+\lambda_{i-1}(N_{Y'}+P_{Y'})+R\big)$-MMP over $Z$ is $(K_{Y'}+B_{Y'}+M_{Y'})$-trivial for any divisor $R$ on $ Y' $ with $\Supp R=\Supp A_{Y'}$ and whose coefficients are smaller than $\varepsilon_0$; this is possible by Proposition \ref{prop:Bir11_prop_g}, and
			
			\item[(c)] any $\big(K_{X_i'}+B_i'+M_i'+\lambda_{i-1}(N_i'+P_i')+\varepsilon A_i'\big)$-MMP over $Z$ is $\big(K_{X_i'}+B_i'+M_i'+\lambda_{i-1}(N_i'+P_i')\big)$-trivial for any $0<\varepsilon\leq \varepsilon_0$; this is again possible by Proposition \ref{prop:Bir11_prop_g}.
		\end{enumerate}
		
		\medskip
	
		\emph{Step 7.} Since the map $X'\dashrightarrow X_i'$ is $\big(K_{X'}+B'+M'+\lambda_{i-1}(N'+P')\big)$-negative, it is also $\big(K_{X'}+B'+M'+\lambda_{i-1}(N'+P')+\varepsilon A'\big)$-negative for $0<\varepsilon\leq\varepsilon_0$, possibly by taking a smaller $\varepsilon_0$. In particular, the g-pair $\big(X_i'/Z,B_i'+M_i'+\lambda_{i-1}(N_i'+P_i')+\varepsilon A_i'\big)$ is $\Q$-factorial dlt, where $A_i'$ is the strict transform of $A'$ on $X_i'$. 
					
		Pick $0<\varepsilon'<\varepsilon_0$. Let $\|\cdot\|$ by the maximum componentwise norm on the space of divisors on $X_i'$. Since the classes of the components of $A_i'$ generate $N^1(X_i'/Z)$, there exist an $\R$-divisor $H$ with $\Supp H=\Supp A_i'$ and $\|H\|\ll \varepsilon_0-\varepsilon'$ which is ample over $Z$. Set $H':=\varepsilon'A_i'-H$. Then $0\leq H'\leq\varepsilon_0 A_i'$. By replacing now $H$ with a general divisor $\R$-linearly equivalent to $H$ and by replacing $\varepsilon_0$ with $\varepsilon'$, we may assume that $H+H'\sim_\R\varepsilon_0A_i'$ and that the g-pair $\big(X_i'/Z, B_i'+M_i'+\lambda_{i-1}(N_i'+P_i')+\varepsilon (H+H')\big)$ is dlt for all $0\leq\varepsilon\leq1$. 
		
		By \cite[Lemma 3.5]{HanLi18} there exists a klt pair $(X_i', \Delta_i')$ such that
		\begin{align*}
			K_{X_i'}+\Delta_i' & \sim_\R K_{X_i'}+B_i'+M_i'+\lambda_{i-1}(N_i'+P_i')+H+H'\\
			& \sim_\R K_{X_i'}+B_i'+M_i'+\lambda_{i-1}(N_i'+P_i')+\varepsilon_0 A_i'.
		\end{align*}
		By \cite{BCHM10} we may run a $(K_{X_i'}+\Delta_i')$-MMP over $Z$ with scaling of an ample divisor: this is clearly a $(K_{X_i'}+B_i'+M_i'+\lambda_{i-1}(N_i'+P_i')+\varepsilon_0 A_i')$-MMP, which terminates with a minimal model $\big(T/Z, B_T+M_T+\lambda_{i-1} (N_T+P_T)+\varepsilon_0 A_T\big)$. Since the divisor $A_i'$ is movable over $Z$, then so is the divisor $K_{X_i'}+B_i'+M_i'+\lambda_{i-1}(N_i'+P_i')+\varepsilon_0 A_i'$, hence
		\begin{equation}\label{eq:7}
			\text{$X_i'$ and $T$ are isomorphic in codimension $1$.}
		\end{equation}
		By (c) we have that
		\begin{equation}\label{eq:8}
			K_T+B_T+M_T+\lambda_{i-1} (N_T+P_T)\text{ is nef over }Z,	
		\end{equation}
		as well as $K_T+B_T+M_T+\lambda_{i-1} (N_T+P_T)+\varepsilon_0 A_T$. Therefore, since the classes of the components of $A_T$ generate $N^1(T/Z)$, there exists a divisor $0 < D_T \leq\varepsilon_0 A_T$ such that 
		\begin{equation}\label{eq:6a}
			K_T+B_T+M_T+\lambda_{i-1} (N_T+P_T)+D_T\text{ is ample over }Z.	
		\end{equation}
		
		\medskip
	
		\emph{Step 8.} 
		Note that $Y'$ and $T$ are isomorphic in codimension $1$ by the assumption in Step 6 and by \eqref{eq:7}. If $D_{Y'}$ is the strict transform of $D_T$ on $Y'$, then
		\begin{equation}\label{eq:88a}
			K_{Y'}+B_{Y'}+M_{Y'}+\lambda_{i-1} (N_{Y'}+P_{Y'})+D_{Y'}\text{ is movable over }Z
		\end{equation}
		by \eqref{eq:6a}.
		
		 By (b) in Step 6 we can run a $\big(K_{Y'}+B_{Y'}+M_{Y'}+\lambda_{i-1}(N_{Y'}+P_{Y'})+D_{Y'}\big)$-MMP over $Z$ with scaling of an ample divisor which is $\big(K_{Y'}+B_{Y'}+M_{Y'}\big)$-trivial. By \eqref{eq:88a}, this MMP consists only of flips. Hence, if $\big(Y'', B_{Y''}+M_{Y''}+\lambda_{i-1} (N_{Y''}+P_{Y''})+D_{Y''})$ is the resulting minimal model over $ Z $, then $Y''$ and $T$ are $\Q$-factorial varieties which are isomorphic in codimension $1$,
	 	\begin{equation}\label{eq:9}
			K_{Y''}+B_{Y''}+M_{Y''}\text{ is nef over }Z,
		\end{equation}
		and
		\begin{equation}\label{eq:88b}
			K_{Y''}+B_{Y''}+M_{Y''}+\lambda_{i-1} (N_{Y''}+P_{Y''})+D_{Y''}\text{ is nef over }Z.	
		\end{equation}
	 	By \cite[Lemma 1.7]{HK00}, \eqref{eq:6a} and \eqref{eq:88b}, the map $T\dashrightarrow Y''$ is a morphism, hence an isomorphism by \cite[Lemma 2.1.4]{Fuj17}. Hence, by \eqref{eq:8} and \eqref{eq:9},
	 	\begin{equation}\label{eq:10}
			K_T+B_T+M_T+\lambda_{i-1} (N_T+P_T) \text{ and } K_T+B_T+M_T \text{ are nef over }Z.
		\end{equation}
		Since $\lambda_{i-1}>\lambda_i>0$ by (a) in Step 6, \eqref{eq:10} implies that also
	 	\begin{equation}\label{eq:11}
			K_T+B_T+M_T+\lambda_i(N_T+P_T)\text{ is nef over }Z.
		\end{equation}
		
		\medskip
	
		\emph{Step 9.} 
		Let $(r,s)\colon V \to X_i'\times T$ be a resolution of indeterminacies. By the Negativity lemma and by \eqref{eq:10} we have
		$$ r^*\big(K_{X_i'}+B_i'+M_i'+\lambda_{i-1} (N_i'+P_i')\big) \sim_\R s^*\big(K_T+B_T+M_T+\lambda_{i-1} (N_T+P_T)\big),$$
		whereas the Negativity lemma and \eqref{eq:11} imply
		$$ r^*\big(K_{X_i'}+B_i'+M_i'+\lambda_i (N_i'+P_i')\big) \sim_\R s^*\big(K_T+B_T+M_T+\lambda_i (N_T+P_T)\big).$$
		Since $\lambda_{i-1}>\lambda_i>0$ by (a) in Step 6, the two relations above yield
		$$ r^*(K_{X_i'}+B_i'+M_i') \sim_\R s^*(K_T+B_T+M_T),$$
		hence $K_{X_i'}+B_i'+M_i'$ is nef over $Z$. This is a contradiction which finishes the proof.
	\end{proof}

	\section{Proofs of main results}
	\label{section:Proofs}
	
	In this section we first obtain analogues of \cite[Proposition 6.2 and Theorem 1.7]{HH20} in the setting of g-pairs, see Proposition \ref{prop:HH19_prop_g} and Theorem \ref{thm:HH19_thm_g}. We prove afterwards the results announced in the introduction.
	
	\begin{prop}\label{prop:HH19_prop_g}
		Assume the existence of minimal models for smooth varieties of dimension $ n-1 $.
		
		Let $ (X/Z,B+M) $ be a $ \Q $-factorial NQC log canonical g-pair of dimension $ n $. Assume that $ (X,B+M) $ has a minimal model in the sense of Birkar-Shokurov over $ Z $ or that $ K_X+B+M $ is not pseudoeffective over $ Z $. Then there exists a $ (K_X + B + M) $-MMP over $Z$ which terminates.
	\end{prop}

	\begin{proof}
		We follow closely the proof of \cite[Proposition 6.2]{HH20}.
		
		\medskip
		
		First, we may run a $ (K_X + B + M) $-MMP over $Z$ such that we obtain a g-pair $(Y,B_Y+M_Y)$ with the property that any $ (K_Y + B_Y + M_Y) $-MMP over $Z$ consists only of flips. In addition to this, either $ (Y,B_Y+M_Y) $ has a minimal model in the sense of Birkar-Shokurov over $ Z $ by \cite[Corollary 3.20]{HaconLiu21} or $ K_Y+B_Y+M_Y $ is not pseudoeffective over $ Z $. 
		Therefore, by replacing $(X,B+M)$ with $(Y,B_Y+M_Y)$, we may assume that any $ (K_X + B + M) $-MMP over $Z$ consists only of flips.
				
		By \cite[Section 3.3]{HanLi18}, we may find positive real numbers $r_1,\dots,r_m$ and $\Q$-divisors $B^{(1)},\dots, B^{(m)}$ and $M^{(1)},\dots, M^{(m)}$ such that 
		$$\sum_{j=1}^m r_j=1, \quad B = \sum_{j=1}^m r_j B^{(j)},\quad M = \sum_{j=1}^m r_j M^{(j)}, $$
		and each g-pair $(X,B^{(j)}+M^{(j)})$ is log canonical. In particular, each divisor $ K_X + B^{(j)} + M^{(j)} $ is $ \Q $-Cartier and we have
		\begin{equation}\label{eq:1_HH19_prop_g}
			K_X + B + M = \sum_{j=1}^m r_j \big(K_X + B^{(j)} + M^{(j)} \big) . 
		\end{equation}
		
		Next, set $d := \dim_\R N^1(X/Z)$ and fix positive real numbers $\alpha_1,\dots,\alpha_d$ which are linearly independent over the field $\Q(r_1,\dots, r_m)$. Pick $\Q$-divisors $A^{(1)},\dots,A^{(d)}$ on $ X $ which are ample over $ Z $ such that their classes form a basis of $N^1(X/Z)$ and such that, setting
		\begin{equation}\label{eq:33}
			A := \alpha_1 A^{(1)}+\dots + \alpha_d A^{(d)},
		\end{equation}
		the g-pair $ \big(X/Z, (B+A)+M \big) $ is log canonical and the divisor $K_X+B+M+A$ is nef over $Z$.
		
		Now, we run a $ (K_X + B + M) $-MMP over $Z$ with scaling of $ A $ -- recall that it consists only of flips:
		\begin{center}
			\begin{tikzcd}[column sep = 0.8em, row sep = 2em]
				(X_1,B_1+M_1) \arrow[dr, "\theta_1" swap] \arrow[rr, dashed, "\pi_1"] && (X_2,B_2+M_2) \arrow[dl, "\theta_1^+"] \arrow[dr, "\theta_2" swap] \arrow[rr, dashed, "\pi_2"] && (X_3,B_3+M_3) \arrow[dl, "\theta_2^+"] \arrow[rr, dashed, "\pi_3"] && \cdots \\
				& Z_1 && Z_2
			\end{tikzcd}
		\end{center}
		where $(X_1,B_1+M_1) := (X,B+M)$. We assume that this MMP does not terminate and we will derive a contradiction. 
		
		To this end, note first that each $X_i$ is $\Q$-factorial, and let $B_i$, $ B_i^{(j)} $, $M_i$, $ M_i^{(j)} $, $A_i$ and $ A_i^{(j)} $ denote the pushforwards on $X_i$ of $B$, $ B^{(j)} $, $M$, $ M^{(j)} $, $A$ and $ A^{(j)} $, respectively. For each positive integer $ i $ set 
		\[ \lambda_i := \inf  \{ t \in \R_{\geq 0} \mid K_{X_i}+B_i + M_i + t A_i \text{ is nef over } Z  \} \]
		and note that $ \lambda_i \geq \lambda_{i+1} $ for every $ i $.
		
		\medskip
		
		We claim that
		\begin{equation}\label{eq:HH19_prop_g}
			\lambda_i > \lambda_{i+1}\quad\text{for every }i.
		\end{equation}
		This will immediately imply the proposition. Indeed, set $\lambda := \lim\limits_{i\to \infty}\lambda_i$ and note that $\lambda < \lambda_i$ for every $ i $ by \eqref{eq:HH19_prop_g}. Since each divisor $ K_{X_i}+B_i + M_i + \lambda_i A_i $ is nef over $ Z $, the divisor $K_X+B+M+\lambda A$ is pseudoeffective over $Z$. It follows from the assumptions of the proposition and from Remark \ref{rem:MMimplWZD} and Lemma \ref{lem:MM_existence_HH19} that the g-pair $\big(X,(B+\lambda A)+M\big)$ has a minimal model in the sense of Birkar-Shokurov over $Z$, hence the above MMP terminates by Theorem \ref{thm:HL18_term_with_scaling}, a contradiction.
		
		\medskip
		
		It remains to prove \eqref{eq:HH19_prop_g}. Assume, for contradiction, that $\lambda_{i}=\lambda_{i+1}$ for some $i$. Pick a curve $C$ on $X_{i+1}$ which is contracted by the flipped contraction $ \theta_i^+\colon X_{i+1}\to Z_i $ and a curve $C'$ on $X_{i+1}$ which is contracted by the flipping contraction $ \theta_{i+1} \colon X_{i+1}\to Z_{i+1}$. The divisors $K_{X_{i+1}}+B_{i+1}^{(j)}+M_{i+1}^{(j)}$ are $\Q$-Cartier, and we have
		\begin{equation}\label{eq:3b}
			(K_{X_{i+1}}+B_{i+1}+M_{i+1})\cdot C > 0\quad\text{and}\quad (K_{X_{i+1}}+B_{i+1}+M_{i+1})\cdot C' < 0.
		\end{equation}
		Hence, by \eqref{eq:1_HH19_prop_g} we obtain
		\begin{equation}\label{eq:3}
			\beta := \frac{(K_{X_{i+1}}+B_{i+1}+M_{i+1})\cdot C}{(K_{X_{i+1}}+B_{i+1}+M_{i+1})\cdot C'}\in \Q(r_1,\dots, r_m)\cap(-\infty,0).
		\end{equation}						
		Since the divisors $A_{i+1}^{(k)}$ are $ \Q $-Cartier, we have 
		\begin{equation}\label{eq:333}
			A_{i+1}^{(k)}\cdot (C-\beta C')\in\Q(r_1,\dots, r_m)\quad\text{for every } k \in \{1,\dots,d\}.
		\end{equation}
			
		By construction of the $ (K_X + B + M) $-MMP over $Z$ with scaling of $ A $ and by \cite[Theorem 1.3(4)(c)]{HaconLiu21}, we have
			$$ (K_{X_{i+1}}+B_{i+1}+M_{i+1}+\lambda_{i}A_{i+1})\cdot C=0 $$
			and
			$$ (K_{X_{i+1}}+B_{i+1}+M_{i+1}+\lambda_{i+1}A_{i+1})\cdot C'=0 . $$
			As $\lambda_{i}=\lambda_{i+1}$, the last two equations and \eqref{eq:3} yield
			$$ A_{i+1} \cdot (C - \beta C') = 0 ,$$
			which, together with \eqref{eq:33}, implies
			$$ \alpha_1A_{i+1}^{(1)}\cdot(C-\beta C')+\dots+\alpha_dA_{i+1}^{(d)}\cdot(C-\beta C')=0 . $$
			Since $ \alpha_1,\dots, \alpha_d$ are linearly independent over $\Q(r_1,\dots, r_m)$, it follows from the last equation and from \eqref{eq:333} that 
		\begin{equation}\label{eq:3a}
			A_{i+1}^{(k)}\cdot(C-\beta C')=0\quad\text{for every } k \in \{1,\dots,d\}.
		\end{equation}
		
		Now, since the maps $(\pi_i)_* \colon N^1(X_i/Z) \to N^1(X_{i+1}/Z)$ are isomorphisms for each $i$, the classes of $A_{i+1}^{(1)}, \dots, A_{i+1}^{(d)}$ form a basis of $N^1(X_{i+1}/Z)$ for all $i$. This, together with \eqref{eq:3a}, implies that the class of $C-\beta C'$ is $0$ in $N_1(X_{i+1}/Z)$. But this is impossible since $\beta<0$. This proves \eqref{eq:HH19_prop_g} and finishes the proof of the proposition.
	\end{proof}
	
	\begin{thm}\label{thm:HH19_thm_g}
		Assume the existence of minimal models for smooth varieties of dimension $ n-1 $.
		
		Let $ \big( X/Z, (B+N) + (M+P) \big) $ be a $ \Q $-factorial NQC log canonical g-pair such that the divisor $ K_X + B + N + M + P $ is nef over $ Z $. Assume that $ (X,B+M) $ has a minimal model in the sense of Birkar-Shokurov over $ Z $ or that $ K_X+B+M $ is not pseudoeffective over $ Z $. Then there exists a $ (K_X + B + M) $-MMP over $Z$ with scaling of $P+N$ that terminates.
	\end{thm}

	\begin{proof}
		We follow closely the proof of \cite[Theorem 1.7]{HH20}.
		
		\medskip
		
		We set $ X_1 := X $, $ B_1 := B $, $ M_1 := M $, $ N_1 := N $, $ P_1 := P $, and 
		\[ \lambda_1 := \inf  \{ t \in \R_{\geq 0} \mid K_{X_1} + B_1 + M_1 + t (N_1 + P_1) \text{ is nef over } Z  \} . \]
		If $\lambda_1=0$, then we are done, so we may assume that $\lambda_1>0$.
		
		Note that if a divisor $ K_{X_1}+B_1+ M_1 + s(N_1+P_1) $ is pseudoeffective over $ Z $ for some $0\leq s\leq\lambda_1$, then by the assumptions of the theorem, Remark \ref{rem:MMimplWZD} and Lemma \ref{lem:MM_existence_HH19}, the g-pair $ \big( X_1, (B_1 + s N_1) + (M_1 + s P_1) \big) $ has a minimal model in the sense of Birkar-Shokurov over $ Z $. By Proposition \ref{prop:Bir11_prop_g} there exists $\xi_1 \in [0,\lambda_1)$ such that any $ \big(K_{X_1}+B_1+ M_1 + \xi_1(N_1+P_1) \big)$-MMP over $Z$ is  $ \big( K_{X_1}+B_1+M_1 + \lambda_1(N_1+P_1) \big)$-trivial, and thus also a $(K_{X_1}+B_1+M_1)$-MMP over $Z$ with scaling of $P_1+N_1$. Therefore, by Proposition \ref{prop:HH19_prop_g} there exists a $\big( K_{X_1}+B_1+M_1+\xi_1(N_1+P_1) \big)$-MMP over $Z$ which terminates either with a minimal model or with a Mori fibre space $\big(X_2,(B_2+\xi_1 N_2)+(M_2+\xi_1 P_2\big)$ of $ \big(X_1,(B_1 + \xi_1 N_1) + (M_1 + \xi_1 P_1) \big) $ over $Z$, and which is also a $\big(K_{X_1}+B_1+M_1+ s(N_1+P_1)\big)$-MMP over $Z$ with scaling of $P_1+N_1$ for all $s\in[0,\xi_1)$. In particular, if for some $0\leq s\leq\xi_1$ a g-pair $ \big(X_1,(B_1 + s N_1) + (M_1 + s P_1) \big) $ has a minimal model in the sense of Birkar-Shokurov over $ Z $, then the g-pair $ \big(X_2,(B_2 + s N_2) + (M_2 + s P_2) \big) $ has a minimal model in the sense of Birkar-Shokurov over $ Z $ by \cite[Corollary 3.20]{HaconLiu21}.
		
		Now, we distinguish two cases. If $ \big(X_2,(B_2 + \xi_1 N_2) + (M_2 + \xi_1 P_2) $ is a Mori fibre space of $ \big(X_1,(B_1 + \xi_1 N_1) + (M_1 + \xi_1 P_1) \big) $ over $Z$, then $ (X_2,B_2+M_2) $ is a Mori fibre space of $ (X_1,B_1+M_1) $ over $Z$ by construction and we are done. 
		Otherwise, the g-pair $ \big(X_2,(B_2 + \xi_1 N_2) + (M_2 + \xi_1 P_2) \big) $ is a minimal model of $ \big(X_1,(B_1 + \xi_1 N_1) + (M_1 + \xi_1 P_1) \big) $ over $Z$, and then we set
		\[ \lambda_2 := \inf  \{ t \in \R_{\geq 0} \mid K_{X_2} + B_2 + M_2 + t (N_2+P_2) \text{ is nef over } Z  \} , \]
		and we observe that $0 \leq \lambda_2 \leq \xi_1 < \lambda_1$ by construction. 
		
		By repeating the above procedure, we obtain a diagram
		$$ (X_1,B_1+M_1) \dashrightarrow (X_2,B_2+M_2) \dashrightarrow (X_3,B_3+M_3) \dashrightarrow \cdots $$
		where each map $ X_i \dashrightarrow X_{i+1} $ is a sequence of steps of a $(K_{X_i}+B_i+M_i)$-MMP over $Z$ with scaling of $P_i+N_i$, and if $\lambda_i$ is the corresponding nef threshold on $X_i$, then either $\lambda_i=0$ for some index $i$, in which case we are done, or the sequence $\{\lambda_i\}$ is strictly decreasing, in which case we set $\lambda:=\lim\limits_{i\to\infty}\lambda_i$. Then the divisor $K_{X_1}+B_1+M_1+\lambda (N_1 + P_1) $ is pseudoeffective over $Z$, so the g-pair $\big(X_1,(B_1+\lambda N_1)+(M_1 + \lambda P_1)\big)$ has a minimal model in the sense of Birkar-Shokurov over $Z$ by the assumptions of the theorem and by Lemma \ref{lem:MM_existence_HH19}. Consequently, by Theorem \ref{thm:HL18_term_with_scaling} this $ (K_{X_1} + B_1 + M_1) $-MMP over $Z$ with scaling of $ P_1 + N_1 $ terminates.
	\end{proof}
	
	We can now prove quickly all the results announced in the introduction.

	\begin{proof}[Proof of Theorem \ref{thm:mainthm}]
		By assumption, either $ K_X+B+M $ is not pseudoeffective over $ Z $ or $ (X,B+M) $ admits an NQC weak Zariski decomposition over $ Z $, in which case it has a minimal model in the sense of Birkar-Shokurov over $Z$ by \cite[Theorem 4.4(i)]{LT22}. Thus, the first statement of the theorem follows immediately from Theorem \ref{thm:HH19_thm_g}. The last sentence of the theorem follows by setting $N=0$ and by taking $P$ to be a sufficiently ample $\R$-divisor on $X$ over $Z$. 
	\end{proof}
	
	\begin{proof}[Proof of Theorem \ref{thm:MM_uniruled_g}]
		The g-pair $ (X,B+M) $ has a minimal model in the sense of Birkar-Shokurov over $Z$ by \cite[Theorem 4.3]{LT22}. We conclude by Remark \ref{rem:MMimplWZD} and by Theorem \ref{thm:mainthm}.
	\end{proof}
	
	\begin{proof}[Proof of Corollary \ref{cor:MM_bigger_boundary_g}]
	The proof is the same as that of Lemma \ref{lem:MM_bigger_boundary_g}, once we invoke Theorem \ref{thm:mainthm} instead of \cite[Theorem 4.4(i)]{LT22}.
	\end{proof}
	
	\begin{proof}[Proof of Corollary \ref{cor:MM_impl_g}]
		The result follows immediately from \cite[Lemma 2.15]{LT22} and Corollary \ref{cor:MM_bigger_boundary_g}.
	\end{proof}
	
	\begin{proof}[Proof of Corollary \ref{cor:maincor}]
		The existence of minimal models for terminal $ 4 $-folds over $ Z $ follows from \cite[Theorem 5-1-15]{KMM87}. Therefore, Remark \ref{rem:MMimplWZD} and \cite[Theorem E]{LT22} imply the existence of NQC weak Zariski decompositions for pseudoeffective NQC log canonical g-pairs of dimension $ 4 $. Hence, (i) and (ii) follow from Theorem \ref{thm:mainthm}, while (iii) follows from Theorem \ref{thm:MM_uniruled_g}.
	\end{proof}

	\section*{Appendix, written jointly with Xiaowei Jiang}
	\setcounter{section}{1}
	\setcounter{thm}{0}
	\renewcommand\thesection{\Alph{section}}
	
	In this appendix we establish the existence of Mori fibre spaces for non-pseudoeffective $\Q$-factorial NQC log canonical generalised pairs of arbitrary dimension. We present two different proofs below. In the following result we summarise previous knowledge concerning the existence of Mori fibre spaces for non-pseudoeffective generalised pairs. 
	
	\begin{thm}
		Let $ (X/Z,B+M) $ be an NQC log canonical g-pair such that $ K_X+B+M$ is not pseudoeffective over $ Z $. Then the following statements hold:
		\begin{enumerate}
			\item If $ M=0 $, then the pair $ (X,B) $ has a Mori fibre space over $ Z $.
			
			\item If $ (X,0) $ is $ \Q $-factorial klt, then the g-pair $ (X,B+M) $ has a Mori fibre space over $ Z $.
		\end{enumerate}
	\end{thm}
	
	\begin{proof}
		Part (i) follows from \cite[Theorem 1.7]{HH20}; note that the $\Q$-factorial dlt case was established in \cite[Corollary 1.3.3]{BCHM10}, see also \cite[Theorem 4.1(ii)]{Bir12a}. Part (ii) is \cite[Lemma 4.4(1)]{BZ16}.
	\end{proof}

	The first approach to derive the existence of Mori fibre spaces for non-pseudoeffective generalised pairs uses very recent results from \cite{Hash22,LX22}. The next theorem is a special case of \cite[Theorem 3.17]{Hash22} and \cite[Theorem 1.2]{LX22}.
	
	\begin{thmapp} \label{thm:HLX}
		Let $ \big(X/Z,(B+A)+M\big) $ be a $ \Q $-factorial NQC log canonical g-pair such that $ K_X + B + A + M $ is pseudoeffective over $ Z $, where $A$ is an effective $ \R $-Cartier $ \R $-divisor on $ X $ which is ample over $ Z $. Then the g-pair $ \big(X/Z,(B+A)+M\big) $ has a minimal model in the sense of Birkar-Shokurov over $ Z $.
	\end{thmapp}
	
	We will give a short proof of Theorem \ref{thm:HLX} below, which depends only on the results from \cite{HH20,HaconLiu21}.
	
	\medskip
		
	We deduce now the following strengthening of Proposition \ref{prop:HH19_prop_g}, which removes the assumption in lower dimensions.
		
	\begin{prop}\label{prop:HH19_prop_g_appendix}
		Let $ (X/Z,B+M) $ be a $ \Q $-factorial NQC log canonical g-pair. Assume that $ (X,B+M) $ has a minimal model in the sense of Birkar-Shokurov over $ Z $ or that $ K_X+B+M $ is not pseudoeffective over $ Z $. Then there exists a $ (K_X + B + M) $-MMP over $Z$ which terminates.
	\end{prop}
	
	\begin{proof}
		For brevity we only indicate here the necessary modifications to the proof of Proposition \ref{prop:HH19_prop_g} and we also use the same notation. 
		
		By arguing by contradiction and by repeating verbatim the proof of Proposition \ref{prop:HH19_prop_g}, we infer that eventually there exists a $ (K_X+B+M) $-MMP with scaling of $ A $ over $ Z $ which consists only of flips, satisfies $ \lambda_i > \lambda_{i+1} $ for every $ i \geq 1 $, but does not terminate by assumption. Setting $\lambda := \lim\limits_{i\to \infty}\lambda_i $, we have $\lambda < \lambda_i$ for every $ i \geq 1$ and the divisor $K_X+B+M+\lambda A$ is pseudoeffective over $Z$. It follows from the assumptions of the proposition when $ \lambda = 0 $ or from Theorem \ref{thm:HLX} when $ \lambda > 0 $ that the g-pair $\big(X,(B+\lambda A)+M\big)$ has a minimal model in the sense of Birkar-Shokurov over $Z$, hence the above MMP terminates by Theorem \ref{thm:HL18_term_with_scaling}, a contradiction which completes the proof.
	\end{proof}
	
	Using the above results, we may finally prove Theorem \ref{thm:Mfs}.
		
	\begin{proof}[Proof of Theorem \ref{thm:Mfs}]
		The proof of this statement is essentially the same as that of Theorem \ref{thm:HH19_thm_g}, except that we replace Lemma \ref{lem:MM_existence_HH19} with Theorem \ref{thm:HLX} and Proposition \ref{prop:HH19_prop_g} with Proposition \ref{prop:HH19_prop_g_appendix}.
	\end{proof}	
	
	After we discussed the above proof with Jihao Liu, he suggested an alternative proof of Theorem \ref{thm:Mfs}(b), which is presented below. We would like to thank him for communicating this proof to us and for allowing us to include it with some more details in this appendix. 
	
	We start with the following easy lemma.
	
	\begin{lem}\label{lem:Nklt}
		Let $ (X/Z,B+M) $ be a $ \Q $-factorial NQC log canonical g-pair. Then for each $\varepsilon\in[0,1)$ we have $\Nklt(X,B+\varepsilon M)=\Nklt(X,B)$.
	\end{lem}
	
	\begin{proof}
		Let $f\colon X' \to X$ be a log resolution of $(X,B)$ such that there exists an NQC divisor $M'$ over $Z$ such that $M=f_*M'$. By the Negativity lemma the $ \R $-divisor $F:= f^*M -M'$ is effective and $f$-exceptional. Then an easy calculation shows that for each prime divisor $E$ on $X'$ and for each $\varepsilon\in[0,1]$ we have
		\begin{equation}\label{eq:a1}
		a(E,X,B)=a(E,X,B+\varepsilon M)+\varepsilon\mult_E F
		\end{equation}
		and
		\begin{equation}\label{eq:a2}
		a(E,X,B+\varepsilon M) = a(E,X,B+M) + (1-\varepsilon) \mult_E F .		
		\end{equation}
		Note that $a(E,X,B+\varepsilon M)\geq{-}1$ for each $ \varepsilon \in [0,1] $ by Lemma \ref{lem:sing_smaller_boundary}. Thus, if $a(E,X,B)={-}1$, then $a(E,X,B+\varepsilon M)={-1}$ by \eqref{eq:a1}. Conversely, if $a(E,X,B+\varepsilon M)={-1}$, then $a(E,X,B+M)={-1}$ and $\mult_EF=0$ by \eqref{eq:a2}, hence $a(E,X,B)={-1}$ by \eqref{eq:a1}. This yields the statement.
	\end{proof}
	
	Next, we obtain:
	
	\begin{lem}\label{lem:hilfe2}
		Let $ \big(X/Z,(B+A)+M\big) $ be a $ \Q $-factorial NQC log canonical g-pair, where $A$ is an effective $ \R $-Cartier $ \R $-divisor on $ X $ which is ample over $ Z $. Then there exists an effective $\R$-divisor $\Delta$ on $ X $ such that $(X,\Delta)$ is a log canonical pair and $ K_X+\Delta \sim_{\R,Z} K_X+B+A+M $.
	\end{lem}
	
	\begin{proof}
		Pick $\varepsilon \in (0,1)$ such that the $ \R $-divisor $A+(1-\varepsilon)M$ is ample over $Z$, and set $N:=\varepsilon M$. Since the g-pair $(X,B+M)$ is log canonical by Lemma \ref{lem:sing_smaller_boundary}, we have
		\begin{equation}\label{eq:699}
			\Nklt(X,B+N)=\Nklt(X,B)
		\end{equation}
		by Lemma \ref{lem:Nklt}. Now, pick a general effective $ \R $-divisor $H\sim_{\R,Z}A+(1-\varepsilon)M$ such that the g-pair $ \big(X,(B+H)+N\big) $ is log canonical, and observe that
		\begin{equation}\label{eq:699b}
			K_X+B+H+N \sim_{\R,Z} K_X+B+A+M .
		\end{equation}
	
		Let $f \colon X' \to X$ be a birational morphism such that there exists an NQC divisor $N'$ over $Z$ such that $N=f_*N'$. By the Negativity lemma the divisor $F:= f^*N -N'$ is effective, and by \cite[Lemma 5.18]{HaconLiu21} we may assume that
		\begin{equation}\label{eq:699a}
		\Exc(f) = \Supp F.
		\end{equation}
		 From now on we argue as in the proof of \cite[Theorem 5.2]{HaconLiu21}. Write $ K_{X'} + B' := f^*(K_X + B) $, so that
		\[ K_{X'} + B' +F + N' = f^*(K_X + B + N) .	\]
		By \eqref{eq:699} and \eqref{eq:699a} we conclude that no $f$-exceptional prime divisor is a log canonical place of $(X,B+N)$. This, together with Lemma \ref{lem:sing_smaller_boundary}, implies that there exists an effective $f$-exceptional $ \R $-divisor $E$ on $X'$ such that the pair $(X',B' + F + E)$ is sub-log canonical and the $ \R $-divisor $f^*H - E$ is ample over $Z$. In particular, the $ \R $-divisor $N' + f^*H - E$ is ample over $Z$, and we now pick a general effective $\R$-divisor $ H' \sim_{\R,Z} N' + f^*H - E$ such that the pair $(X',B' + F + E + H')$ is sub-log canonical. Note that 
		\begin{equation}\label{eq:699c}
			K_{X'} + B' +F + E + H' \sim_{\R,Z} f^*(K_X + B + H + N) ,
		\end{equation}
		so there exists an $\R$-divisor $D$ on $X$ such that $K_{X'} + B' +F + E + H' \sim_\R f^*D$. Set $\Delta := B + f_*H' $. Then $K_X+\Delta\sim_\R D$, so that $K_{X'} + B' +F + E + H' \sim_\R f^*(K_X+\Delta)$, and hence the pair $(X,\Delta)$ is log canonical. Since 
		\[ K_X + \Delta \sim_{\R,Z} K_X+B+A+M , \]
		by \eqref{eq:699b} and \eqref{eq:699c}, this finishes the proof.
	\end{proof}
	
	As promised, we give now a short proof of Theorem \ref{thm:HLX}.

	\begin{proof}[Proof of Theorem \ref{thm:HLX}]
		Pick $ \xi \in (0,1) $ and note that $ \big(X/Z,(B+\xi A)+M\big) $ is a $ \Q $-factorial NQC log canonical g-pair by Lemma \ref{lem:sing_smaller_boundary}. By Lemma \ref{lem:hilfe2} there exists an effective $\R$-divisor $\Delta$ on $ X $ such that $(X,\Delta)$ is a log canonical pair and 
		\[ K_X+\Delta \sim_{\R,Z} K_X+B+\xi A+M . \] 
		Pick a general ample over $Z$ effective $ \R $-divisor $ H \sim_{\R,Z} A $ such that the pair $ \big( X/Z, \Delta + (1 - \xi) H \big) $ is log canonical, and observe that the divisor 
		\[ K_X+\Delta + (1 - \xi) H \sim_{\R,Z} K_X + B + A + M \]
		is pseudoeffective over $ Z $. By \cite[Theorem 1.5]{HH20} this pair has a minimal model in the sense of Birkar-Shokurov over $ Z $, thus it has a minimal model over $Z$ by \cite[Theorem 1.7]{HH20}. Hence, the g-pair $ \big(X,(B+A)+M\big) $ has a minimal model over $ Z $.
	\end{proof}
		
	Finally, we can give an alternative proof of Theorem \ref{thm:Mfs}(b).
	
	\begin{proof}[Proof of Theorem \ref{thm:Mfs}(b)]
		By assumption, the divisor $ K_X+B+M $ is not pseudoeffective over $ Z $. Pick a general ample over $Z$ $ \R $-divisor $A \geq 0$ on $X$. Then for $0<\varepsilon\ll1$ the divisor $K_X+B+\varepsilon A+M$ is not pseudoeffective over $ Z $, and it suffices to show that there exists a $(K_X+B+\varepsilon A+M)$-MMP with scaling of $A$ which terminates. By Lemma \ref{lem:hilfe2} there exists an effective $\R$-divisor $\Delta$ on $X$ such that $(X,\Delta)$ is a log canonical pair and $K_X+\Delta \sim_{\R,Z} K_X+B+\varepsilon A+M$. We conclude by \cite[Theorem 1.7]{HH20}.
	\end{proof}	
	
	\bibliographystyle{amsalpha}
	\bibliography{BibliographyForPapers}

\providecommand{\bysame}{\leavevmode\hbox to3em{\hrulefill}\thinspace}
\providecommand{\MR}{\relax\ifhmode\unskip\space\fi MR }
\providecommand{\MRhref}[2]{%
  \href{http://www.ams.org/mathscinet-getitem?mr=#1}{#2}
}
\providecommand{\href}[2]{#2}
\begin{thebibliography}{BCHM10}

\bibitem[BCHM10]{BCHM10}
C.~Birkar, P.~Cascini, C.~D. Hacon, and J.~M\textsuperscript{c}Kernan,
  \emph{Existence of minimal models for varieties of log general type}, J.
  Amer. Math. Soc. \textbf{23} (2010), no.~2, 405--468.

\bibitem[Bir11]{Bir11}
C.~Birkar, \emph{On existence of log minimal models {II}}, J. Reine Angew.
  Math. \textbf{658} (2011), 99--113.

\bibitem[Bir12a]{Bir12a}
\bysame, \emph{Existence of log canonical flips and a special {LMMP}}, Publ.
  Math. Inst. Hautes \'Etudes Sci. \textbf{115} (2012), no.~1, 325--368.

\bibitem[Bir12b]{Bir12b}
\bysame, \emph{On existence of log minimal models and weak {Z}ariski
  decompositions}, Math. Ann. \textbf{354} (2012), no.~2, 787--799.

\bibitem[Bir21]{Bir21c}
\bysame, \emph{Generalised pairs in birational geometry}, EMS Surv. Math. Sci.
  \textbf{8} (2021), no.~1-2, 5--24.

\bibitem[BZ16]{BZ16}
C.~Birkar and D.-Q. Zhang, \emph{Effectivity of {I}itaka fibrations and
  pluricanonical systems of polarized pairs}, Publ. Math. Inst. Hautes \'Etudes
  Sci. \textbf{123} (2016), 283--331.

\bibitem[CT20]{CT20}
G.~Chen and N.~Tsakanikas, \emph{On the termination of flips for log canonical
  generalized pairs}, arXiv:2011.02236, to appear in Acta Math. Sin. Engl.
  Ser\setbox0=\hbox{2020}.

\bibitem[Fil20]{Fil20}
S.~Filipazzi, \emph{On a generalized canonical bundle formula and generalized
  adjunction}, Ann. Sc. Norm. Super. Pisa Cl. Sci \textbf{21} (2020),
  1187--1221.

\bibitem[Fuj17]{Fuj17}
O.~Fujino, \emph{Foundations of the minimal model program}, MSJ Memoirs,
  vol.~35, Mathematical Society of Japan, Tokyo, 2017.

\bibitem[Has20]{Hash20b}
K.~Hashizume, \emph{Non-vanishing theorem for generalized log canonical pairs
  with a polarization}, arXiv:2012.15038\setbox0=\hbox{2020}.

\bibitem[Has22]{Hash22}
\bysame, \emph{Iitaka fibrations for dlt pairs polarized by a nef and log big
  divisor}, arXiv:2203.05467\setbox0=\hbox{2022}.

\bibitem[HH20]{HH20}
K.~Hashizume and Z.-Y. Hu, \emph{On minimal model theory for log abundant lc
  pairs}, J. Reine Angew. Math. \textbf{767} (2020), 109--159.

\bibitem[HK00]{HK00}
Y.~Hu and S.~Keel, \emph{Mori dream spaces and {GIT}}, Michigan Math. J.
  \textbf{48} (2000), 331--348.

\bibitem[HL18]{HanLi18}
J.~Han and Z.~Li, \emph{{Weak Zariski decompositions and log terminal models
  for generalized polarized pairs}}, arXiv:1806.01234, to appear in Math.
  Z\setbox0=\hbox{2018}.

\bibitem[HL20]{HanLiu20}
J.~Han and W.~Liu, \emph{On numerical nonvanishing for generalized log
  canonical pairs}, Doc. Math. \textbf{25} (2020), 93--123.

\bibitem[HL21a]{HaconLiu21}
C.~D. Hacon and J.~Liu, \emph{Existence of flips for generalized lc pairs},
  arXiv:2105.13590\setbox0=\hbox{2021}.

\bibitem[HL21b]{HanLiu21}
J.~Han and W.~Liu, \emph{On a generalized canonical bundle formula for
  generically finite morphisms}, Ann. Inst. Fourier (Grenoble) \textbf{71}
  (2021), no.~5, 2047--2077.

\bibitem[HM20]{HM20}
C.~D. Hacon and J.~Moraga, \emph{On weak {Z}ariski decompositions and
  termination of flips}, Math. Res. Lett. \textbf{27} (2020), no.~5,
  1393--1422.

\bibitem[Hu20]{Hu20}
Z.~Hu, \emph{Log abundance of the moduli b-divisors of lc-trivial fibrations},
  arXiv:2003.14379\setbox0=\hbox{2020}.

\bibitem[HX13]{HX13}
C.~D. Hacon and C.~Xu, \emph{Existence of log canonical closures}, Invent.
  Math. \textbf{192} (2013), no.~1, 161--195.

\bibitem[KM98]{KM98}
J.~Koll{\'a}r and S.~Mori, \emph{Birational geometry of algebraic varieties},
  Cambridge Tracts in Mathematics, vol. 134, Cambridge University Press,
  Cambridge, 1998.

\bibitem[KMM87]{KMM87}
Y.~Kawamata, K.~Matsuda, and K.~Matsuki, \emph{Introduction to the minimal
  model problem}, Algebraic geometry, {S}endai, 1985, Adv. Stud. Pure Math.,
  vol.~10, North-Holland, Amsterdam, 1987, pp.~283--360.

\bibitem[Les16]{Les16}
J.~Lesieutre, \emph{A pathology of asymptotic multiplicity in the relative
  setting}, Math. Res. Lett. \textbf{23} (2016), no.~5, 1433--1451.

\bibitem[LMT20]{LMT}
V.~Lazi\'c, J.~Moraga, and N.~Tsakanikas, \emph{Special termination for log
  canonical pairs}, arXiv:2007.06458\setbox0=\hbox{2020}.

\bibitem[LP18]{LP18}
V.~Lazi\'c and Th. Peternell, \emph{Abundance for varieties with many
  differential forms}, {{\'E}pijournal de G{\'e}om{\'e}trie Alg{\'e}brique}
  \textbf{2} (2018).

\bibitem[LT22]{LT22}
V.~Lazi\'c and N.~Tsakanikas, \emph{On the existence of minimal models for log
  canonical pairs}, Publ. Res. Inst. Math. Sci. \textbf{58} (2022), no.~2,
  311--339.

\bibitem[LX22]{LX22}
J.~Liu and L.~Xie, \emph{{Relative Nakayama-Zariski decomposition and minimal
  models of generalized pairs}}, arXiv:2207.09576\setbox0=\hbox{2022}.

\bibitem[Mor18]{Mor18}
J.~Moraga, \emph{{Termination of pseudo-effective 4-fold flips}},
  arXiv:1802.10202\setbox0=\hbox{2018}.

\bibitem[Nak04]{Nak04}
N.~Nakayama, \emph{Zariski-decomposition and abundance}, MSJ Memoirs, vol.~14,
  Mathematical Society of Japan, Tokyo, 2004.

\end{thebibliography}
	
\end{document}